\documentclass[10pt]{article}
\usepackage{amsmath, amssymb, amsthm, amsfonts, cases}
\usepackage{mathrsfs}
\usepackage{url}
\usepackage{authblk}
\usepackage[usenames]{color}
\usepackage{geometry}
\usepackage{indentfirst}
\allowdisplaybreaks[4]
\DeclareMathOperator*{\essinf}{ess\,inf}

\theoremstyle{plain}
\newtheorem{thm}{Theorem}[section]

\newtheorem{lem}[thm]{Lemma}
\newtheorem{prop}[thm]{Proposition}

\theoremstyle{definition}
\newtheorem{defn}{Definition}[section]
\newtheorem{ass}{Assumption}[section]
\newtheorem{rmk}{Remark}[section]

\makeatletter\@addtoreset{equation}{section} \makeatother

\begin{document}		
	\title{ Viscosity Solutions of Hamilton--Jacobi--Bellman Equations for Control Systems Driven by Teugels Martingales
		\thanks{Q. Meng was supported  the National Natural Science Foundation of China ( No.12271158)}}
	
	\date{}
	
	\author{Yongpeng Lin}
	\author{Qingxin Meng\footnote{Corresponding author.
			\authorcr
			\indent E-mail address: YongpengLin1725@163.com (Y. Lin), tmorning@zjhu.edu.cn(M. Tang), mqx@zjhu.edu.cn (Q.Meng)}}
	\author{Maoning Tang}
	
	\affil{\small{Department of Mathematical Sciences, Huzhou University, Zhejiang 313000,PR  China}}
	
	\maketitle

	\begin{abstract}
		In this paper, we study a class of stochastic optimal control problems driven by L\'{e}vy processes. The controlled state process is characterized by a forward--backward stochastic differential equation (FBSDE), in which the jump perturbations are introduced through a system of Teugels martingales to describe the random features of general L\'{e}vy processes. Based on the comparison theorem and stability theory for backward stochastic differential equations (BSDEs) driven by L\'{e}vy processes, the corresponding dynamic programming principle is established. Consequently, the Hamilton--Jacobi--Bellman (HJB) equation associated with the value function is derived, and it is shown within the viscosity solution framework that the value function is a viscosity solution of this equation. Since existing methods in the literature are difficult to extend directly to general L\'{e}vy-driven control problems, we prove the uniqueness of the viscosity solution to the HJB equation by means of the comparison principle, which in turn characterizes the uniqueness of the value function.	
	\end{abstract}
	

	\section{Introduction}
	
	Stochastic optimal control problems stand as a cornerstone of modern control theory and applied mathematics, with profound implications for engineering, finance, insurance, and other fields featuring uncertain dynamics. Since Bellman proposed the Dynamic Programming Principle (DPP) in the 1950s, this principle has revolutionized the analysis of global optimal control by leveraging the time consistency of the value function to reduce complex control problems to the solvability of associated Hamilton--Jacobi--Bellman (HJB) equations. In stochastic settings, the value function typically lacks classical smoothness, rendering the HJB equation intractable under traditional solution frameworks---thus elevating viscosity solution theory to a unified and indispensable analytical tool for characterizing value functions in stochastic optimal control.

The nonlinear theory of backward stochastic differential equations (BSDEs), established by Pardoux and Peng \cite{pardoux1990adapted}, has significantly advanced stochastic optimal control by providing a natural probabilistic representation of recursive cost functionals in controlled stochastic systems. Building on this theory, Peng \cite{peng1997backward} introduced the backward semigroup method, which reveals a fundamental connection between BSDEs, the dynamic programming principle, and viscosity solutions of HJB equations. We also refer to the monograph of Yong and Zhou \cite{yong1999stochastic} for a systematic treatment of stochastic control theory in the continuous diffusion setting, including the maximum principle and the theory of Hamilton--Jacobi--Bellman equations.

However, the above theoretical framework is primarily developed in the continuous diffusion setting and does not directly extend to systems with discontinuous dynamics. From the perspective of stochastic analysis, L\'{e}vy process-driven jump perturbations introduce structural complexities that are fundamentally different from those in Brownian-motion-based systems. A general L\'{e}vy process contains infinitely many random jumps, and its discontinuous sample paths prevent the direct application of the classical martingale representation theory associated with finite-dimensional Brownian motion.

A critical theoretical challenge therefore arises: within the BSDE framework, how can one construct martingale representations adapted to the jump structure of L\'{e}vy processes so as to ensure well-posedness, stability estimates, and rigorous stochastic control analysis? Teugels addressed this issue by introducing an orthogonalization approach based on the power-jump processes of L\'{e}vy processes. This idea was further developed by Nualart and Schoutens \cite{nualart2000chaotic} into a systematic theory of Teugels martingales for square-integrable L\'{e}vy processes, leading to a fundamental martingale representation theorem.

This framework has enabled substantial progress in the study of L\'{e}vy-driven BSDEs and their connections with integro-partial differential equations \cite{nualart2001backward, elotmani2008backward}, and provides a natural foundation for the analysis of stochastic control problems with general jump dynamics.

While Barles and Imbert \cite{barles2008second} systematically developed viscosity solution theory for second-order integro-differential equations with L\'{e}vy-type nonlocal operators, most of the existing literature on stochastic optimal control for L\'{e}vy-driven systems is formulated within the framework of the Maximum Principle (MP), focusing primarily on necessary or sufficient optimality conditions rather than on the dynamic programming principle and the viscosity characterization of the value function \cite{meng2009necessary, hafayed2015mean, meherrem2019maximum, bougherara2018maximum, hafayed2016partial, tang2012optimal}.

As a consequence, although significant progress has been made in the MP-based analysis of stochastic control problems with jumps, a systematic treatment of such problems within the dynamic programming framework—particularly one that rigorously connects the DPP with viscosity solutions of the associated HJB equations in the presence of general L\'{e}vy jump structures remains largely incomplete.

Related studies further highlight the limitations of current methodologies. Li and Wei \cite{LiWei2014} and Hu, Ji, and Xue \cite{hu2019existence} investigated FBSDE-based stochastic optimal control in continuous diffusion settings, establishing DPP and proving viscosity solution uniqueness for associated HJB equations. However, their frameworks are confined to Brownian motion-driven systems and exclude jump-type perturbations. Li and Peng \cite{lipeng2009stochastic} extended BSDE methods to jump-disturbed control problems, deriving DPP and verifying viscosity solution properties of HJB equations, but their results are restricted to jump structures jointly driven by Brownian motion and Poisson random measures, excluding the general L\'{e}vy process case with infinite jumps. Although recent works have explored L\'{e}vy-driven FBSDEs and Teugels martingales \cite{nualart2001backward, elotmani2008backward, bahlali2003bsde}, they lack a systematic integration of DPP, viscosity solution theory, and general L\'{e}vy jump dynamics.

This paper makes three main contributions to stochastic optimal control theory.

First, we establish a unified analytical framework for FBSDE-based stochastic optimal control problems driven by general L\'{e}vy processes, using Teugels martingales to handle the infinite jump structure. Unlike existing works limited to Poisson jumps or continuous diffusion, our framework accommodates general L\'{e}vy processes and infinite-dimensional Teugels martingales simultaneously.

Second, we rigorously establish the dynamic programming principle for L\'{e}vy-driven FBSDE control systems. By combining stability estimates of L\'{e}vy-type BSDEs and the orthogonality of Teugels martingales, we build a precise connection between DPP and the viscosity solution theory of the corresponding HJB equations, filling the gap in the existing literature for general L\'{e}vy processes.

Third, we prove that the value function of the control problem is the unique viscosity solution of the associated HJB equation under the Teugels martingale framework.

Existing methods encounter great difficulty in proving uniqueness for general L\'{e}vy-driven systems due to nonlocal jump operators. To overcome this difficulty, we establish a new comparison principle for HJB equations involving infinite-dimensional Teugels martingales, which ensures the uniqueness of the viscosity solution and thus the uniqueness of the value function.

The remainder of this paper is organized as follows. In Section~2, we introduce preliminary notations, the probabilistic setting, and the construction of Teugels martingales. Section~3 presents the stochastic optimal control problem, the associated forward-backward stochastic differential equations, and their well-posedness. In Section~4, we introduce the definition of viscosity solutions to the HJB equation and prove existence via the BSDE approach. Finally,
Section~5 is devoted to the uniqueness of viscosity solutions by applying the comparison principle under the integral formulation of the L\'{e}vy operator on the whole space.

	\section{Notations and Preliminaries}\label{sec:2}
	This section specifies the notation, function spaces, and probabilistic framework used throughout the paper.
	
	Let $\mathbb{R}^n$ denote the $n$-dimensional Euclidean space, $\mathbb{R}^{k\times n}$ the set of all real $k \times n$ matrices, and $\mathbb{S}^n$ the set of all $n \times n$ symmetric matrices. The control set $U \subset \mathbb{R}^k$ is assumed to be a nonempty compact set. For vectors $x,y \in \mathbb{R}^n$ and matrices $M,N \in \mathbb{R}^{k\times n}$, the inner products and norms are defined by
	\[
	\langle x,y \rangle = x^\top y, \quad \|x\| = \sqrt{\langle x,x \rangle},
	\quad \langle M,N \rangle = \mathrm{tr}(MN^\top), \quad \|M\| = \sqrt{\langle M,M \rangle},
	\]
	where $\mathrm{tr}(\cdot)$ denotes the trace of a matrix.
	
	\subsection{Construction of Teugels Martingales and the Martingale Representation Theorem}
	
	Let $(\Omega, \mathcal{F}, \mathbb{P})$ be a complete probability space on which a real-valued L\'{e}vy process $\{L_t\}_{t\in[0,T]}$ is defined, whose sample paths are right-continuous with left limits (c\`{a}dl\`{a}g). The natural filtration generated by this process is defined by
	\[
	\mathcal{F}_t = \sigma\{ L_s: s \le t \} \vee \mathcal{N}, \quad t \in [0,T],
	\]
	where $\mathcal{N}$ denotes the collection of all $\mathbb{P}$-null sets. We denote this filtration by $\mathcal{F} = \{\mathcal{F}_t\}_{t\in[0,T]}$.
	
	According to the L\'{e}vy--Khintchine formula, the characteristic function of the L\'{e}vy process is given by
	\[
	\mathbb{E}[e^{i u L_t}] = e^{t \psi(u)}, \quad u \in \mathbb{R}, \ t \in [0,T],
	\]
	where the characteristic exponent is
	\[
	\psi(u) = i b u - \frac{1}{2} \sigma^2 u^2 + \int_{\mathbb{R}} \left( e^{i u x} - 1 - i u x \mathbf 1_{\{|x|\le 1\}} \right) \nu(dx),
	\]
	Thus $L$ is characterized by its L\'{e}vy triplet $(b,\sigma,\nu)$
	where $b\in\mathbb R$, $\sigma^2\ge 0$ and $\nu$ is a measure defined on $\mathbb R\setminus\{0\}$
	and satisfies
	\begin{itemize}
		\item[(i)]
		\[
		\int_{\mathbb{R}} (1 \wedge |x|^2)\, \nu(dx) < +\infty,
		\]
		\item[(ii)]
		\[
		\exists\, \epsilon > 0 \text{ and } \lambda > 0 \text{ such that }
		\int_{(-\epsilon,\epsilon)^c} e^{\lambda |x|}\, \nu(dx) < +\infty.
		\]
	\end{itemize}
	This implies that the L\'{e}vy measure admits moments of all orders, i.e.
	\[
	\int_{\mathbb{R}} |x|^i \, \nu(dx) < \infty, \quad \forall i \ge 2.
	\]

	Denote by $\Delta L_t = L_t - L_{t-}$ the jump of the L\'{e}vy process at time $t$, $t \in [0,T]$. Define the corresponding power-jump processes by
	\[
	L_t^{(1)} := L_t,\qquad
	L_t^{(i)} := \sum_{0<s\le t}(\Delta L_s)^i,\quad i\ge2,
	\]
	which characterize the higher-order cumulative features of the jumps of the L\'{e}vy process over a given time interval.
	
	According to the L\'{e}vy--It\^{o} decomposition, a L\'{e}vy process can be decomposed into a drift part, a Brownian continuous perturbation part, and a jump part:
	\[
	L_t = b t + \sigma W_t + \int_0^t \int_{|x|\ge 1} x N(ds,dx) + \int_0^t \int_{|x|<1} x \tilde{N}(ds,dx),
	\]
	where $W_t$ is a standard Brownian motion, $N(ds,dx)$ is a Poisson random measure, and $\tilde{N}(ds,dx) = N(ds,dx) - \nu(dx)\, ds$ is the compensated Poisson random measure.
	
	Taking expectations on both sides of the decomposition, we observe that the Brownian motion term has zero mean, and the compensated Poisson integral $\int_0^t \int_{|x|<1} x \tilde{N}(ds,dx)$ is a martingale with zero mean. For the large-jump term, using the fact that $N(ds,dx)$ has intensity measure $\nu(dx)ds$, we obtain
	\[
	\mathbb{E}\left[\int_0^t \int_{|x|\ge 1} x N(ds,dx)\right] = \int_0^t \int_{|x|\ge 1} x \nu(dx) ds = t \int_{|x|\ge 1} x \nu(dx).
	\]
	Therefore,
	\[
	\mathbb{E}[L_t] = b t + 0 + t \int_{|x|\ge 1} x \nu(dx) + 0 = t\left(b + \int_{|x|\ge 1} x \nu(dx)\right).
	\]
	We define the first-order mean parameter $m_1$ as
	\[
	m_1 = \mathbb{E}[L_1]= b + \int_{|x|\ge 1} x \nu(dx),
	\]
	so that $\mathbb{E}[L_t] = \mathbb{E}[L_t^{(1)}] = m_1 t < \infty$, where the finiteness of $m_1$ is guaranteed by the exponential-moment assumption.
	
	For the higher-order power-jump processes $L_t^{(i)} = \sum_{0 < s \le t} (\Delta L_s)^i$ with $i \ge 2$, they are purely discontinuous and their first moments are determined solely by the L\'{e}vy measure. As established in the classical result (see, e.g.,\cite{protter2005stochastic}), the expectation of $L_t^{(i)}$ is given by the integral of $x^i$ with respect to the L\'{e}vy measure over time:
	\[
	\mathbb{E}[L_t^{(i)}] = \mathbb{E}\left[ \sum_{0 < s \le t} (\Delta L_s)^i \right] = t \int_{\mathbb{R}} x^i \, \nu(dx), \quad i \ge 2.
	\]
	This motivates the natural definition of the $i$-th order mean parameter $m_i$ for $i \ge 2$ as:
	\[
	m_i := \int_{\mathbb{R}} x^i \, \nu(dx),
	\]
	so that $\mathbb{E}[L_t^{(i)}] = m_i t$.
	
	Based on the power-jump processes $\{L_t^{(i)}\}$, we define the zero-mean martingales by
	\[
	Y_t^{(i)} = L_t^{(i)} - m_i t, \quad i \ge 1.
	\]
	
	To achieve the orthogonality of $\{Y_t^{(i)}\}$, we introduce the measure
	\[
	\mu(dx) = x^2 \nu(dx) + \sigma^2 \delta_0(dx),
	\]
	and consider the space $S_1$ consisting of all real-valued polynomials defined on the real axis. We equip this space with the following inner product: for any polynomials $f(x)$ and $g(x)$, define
	\[
	\langle f, g \rangle_1 = \int_{\mathbb{R}} f(x) g(x) \mu(dx)
	= \int_{\mathbb{R}} f(x) g(x) x^2 \nu(dx) + \sigma^2 f(0) g(0).
	\]
	Within the framework of the above inner product $\langle\cdot,\cdot\rangle_1$, we apply the Gram--Schmidt orthogonalization procedure to the polynomial sequence $\{1, x, x^2, \ldots\}$ and construct a family of orthogonal polynomials $\{q_n(x)\}_{n\ge 0}$ such that, for $n \ge 1$, the polynomial $q_{n-1}$ admits the representation
	\[
	q_{n-1}(x)=\sum_{j=1}^{n} a_{n j} x^{j-1},
	\]
	where
	\[
	q_0(x) \equiv a_{1,1}
	= \frac{1}{\sqrt{\int_{\mathbb{R}} x^2 \nu(dx) + \sigma^2}},
	\qquad
	\langle q_n, q_m \rangle_1 = \delta_{nm}.
	\]

	Define
	\[
	p_n(x) = x q_{n-1}(x)
	= a_{n,n} x^n + a_{n,n-1} x^{n-1} + \cdots + a_{n,1} x,
	\quad
	H_t^{(i)} = \sum_{j=1}^i a_{ij} Y_t^{(j)},
	\]
	which yields a family of Teugels martingales satisfying the strong orthogonality property:
	\[
	\langle H^{(i)}, H^{(j)} \rangle_t = \delta_{ij} t,
	\]
	where $\langle \cdot,\cdot\rangle$ denotes the predictable quadratic covariation.

	In particular, when $i=1$, we have
	\[
	H_t^{(1)} = a_{11} (L_t - m_1 t), \quad dL_t = \frac{1}{a_{11}} dH_t^{(1)} + m_1 dt.
	\]
	
	Let us introduce the following spaces:
	
	\begin{itemize}
		\item $\ell^2 = \{ x = (x_n)_{n \ge 1} \,;\, \|x\|_{\ell^2} = \left(\sum_{n=1}^{\infty} |x_n|^2\right)^{\frac12} < \infty \}.$
		
		\item $L^2(\Omega,\mathcal{F}_T,\mathbb{P};\mathbb{R}) = \{ \xi : \Omega \to \mathbb{R} \mid \xi \text{ is } \mathcal{F}_T\text{-measurable and } \mathbb{E}[|\xi|^2] < \infty \}.$
		
		\item For any $0 \le t \le T$, we define the space of square-integrable processes on $[t,T]$ by
		\[
		\mathcal{H}^2(t,T;\mathbb{R}) = \Big\{ \varphi : [t,T] \times \Omega \to \mathbb{R} \mid \varphi \text{ is } \mathcal{F}\text{-progressively measurable and } \mathbb{E} \int_{t}^{T} |\varphi_s|^{2} \, ds < \infty \Big\}.
		\]
		
		\item $\mathcal{P}^2(t,T;\mathbb{R})$ denotes the subspace of $\mathcal{H}^2(t,T;\mathbb{R})$ consisting of predictable processes.
		
		\item $\mathcal{S}^2(t,T;\mathbb{R}) = \Big\{ \varphi : [t,T] \times \Omega \to \mathbb{R} \mid \varphi \text{ is } \mathcal{F}\text{-adapted, c\`{a}dl\`{a}g and } \mathbb{E} \big[ \sup_{s \in [t,T]} |\varphi_s|^{2} \big] < \infty \Big\}.$
		
		\item For $\ell^2$-valued processes, we define
		\[
		\mathcal{H}^2(t,T;\ell^{2}) = \Big\{ \phi = (\phi^{(k)})_{k \ge 1} : [t,T] \times \Omega \to \ell^2 \mid \phi \text{ is } \mathcal{F}\text{-progressively measurable and } \mathbb{E} \int_{t}^{T} \|\phi_s\|_{\ell^{2}}^2 \, ds < \infty \Big\},
		\]
		and $\mathcal{P}^2(t,T;\ell^{2})$ denotes its predictable subspace. The norm on $\mathcal{H}^2(t,T;\ell^{2})$ is given by
		\[
		\|\phi\|_{\mathcal{H}^2(t,T;\ell^{2})}^{2} = \mathbb{E} \int_{t}^{T} \|\phi_s\|_{\ell^{2}}^2 \, ds.
		\]
	\end{itemize}
	
	Under the above function spaces and probabilistic framework, by Theorem~2 in Nualart and Schoutens \cite{nualart2000chaotic}, we obtain the following Teugels martingale representation theorem.
	
	\begin{lem}
		Under the above framework, any square-integrable $\mathcal{F}_T$-measurable random variable $F$ admits the unique representation
		\[
		F = \mathbb{E}[F] + \sum_{k=1}^{\infty} \int_0^T Z_s^{(k)} \, dH_s^{(k)},
		\quad
		Z = (Z_s^{(k)})_{k \ge 1} \in \mathcal{P}^2(0,T;\ell^2).
		\]
	\end{lem}
	
	This representation lays the foundation for constructing solutions of L\'{e}vy-driven FBSDEs via Teugels martingales and for further analyzing their generator operators and the associated stochastic optimal control problems.
	
	\subsection{Backward stochastic differential equations (BSDEs) with jumps}
	
	We consider the following backward stochastic differential equation (BSDE) with jumps:
	\begin{equation}\label{eq:2.1}
		Y_t=
		\eta
		+
		\int_t^T
		f\bigl(s, Y_{s}, Z_s\bigr)\,ds
		-
		\sum_{k=1}^{\infty}
		\int_t^T
		Z_s^{(k)}\, dH_s^{(k)},
		\qquad t\in[0,T].
	\end{equation}
	Here $T>0$ is a given terminal time, and
	$f:\Omega\times[0,T]\times\mathbb{R}\times\ell^2\to\mathbb{R}$
	is the generator function.
	For any $(y,z)\in\mathbb{R}\times\ell^2$,
	$f(\cdot,\cdot,y,z)$ is progressively measurable,
	and satisfies the following assumptions:
	
	\begin{ass}\label{ass:2.1}
		\leavevmode
		\begin{itemize}
			\item[(i)]
			There exists a constant $C \ge 0$ such that, $\mathbb{P}$-a.s., for any
			$t \in [0,T]$,
			$y_1,y_2 \in \mathbb{R}$,
			$z_1,z_2 \in \ell^2$, we have
			\[
			\bigl|
			f(t,y_1,z_1)-f(t,y_2,z_2)
			\bigr|
			\le
			C\bigl(
			|y_1-y_2|
			+
			\|z_1-z_2\|_{\ell^2}
			\bigr).
			\]
			
			\item[(ii)]
			\[
			\mathbb{E}\int_0^T |f(s,0,0)|^2 ds < \infty.
			\]
		\end{itemize}
	\end{ass}
	
	Under the above assumptions, by Theorem~3.1 in \cite{bahlali2003bsde}, we obtain the following classical result.
	
	\begin{lem}\label{lem:2.2}
		Under Assumption~\ref{ass:2.1},
		for any
		$\eta \in L^2(\Omega,\mathcal{F}_T,\mathbb{P};\mathbb{R})$,
		equation \eqref{eq:2.1}
		admits a unique $\mathcal{F}_t$-adapted solution $(Y,Z)$
		such that
		\[
		(Y,Z)
		\in
		\mathcal{S}^2(0,T;\mathbb{R})
		\times
		\mathcal{P}^2(0,T;\ell^2).
		\]
		Moreover,
		\[
		\mathbb{E}\Bigl[\sup_{0\le t\le T}|Y_t|^2\Bigr] < \infty,
		\qquad
		\mathbb{E}\int_0^T \|Z_s\|_{\ell^2}^2 \, ds < \infty.
		\]
	\end{lem}

	We now present a comparison theorem for BSDE \eqref{eq:2.1} with generators linear in $z$
	(see Theorem~3.1 in \cite{elotmani2008backward}).

	\begin{lem}[Comparison]\label{lem:2.3}
		Let $\eta,\eta'\in L^{2}(\Omega,\mathcal{F}_T,\mathbb{P};\mathbb{R})$.
		Let $(Y,Z)$ (resp.\ $(Y',Z')$) be the unique solution of \eqref{eq:2.1}
		with data $(\eta,f)$ (resp.\ $(\eta',f')$).
		
		Assume that both generators $f$ and $f'$ are predictable and admit the decompositions
		\begin{equation}\label{eq:lem23-decomp}
			\left\{
			\begin{aligned}
				f(t,\omega,y,z)
				&=
				f^{1}(t,\omega,y)
				+\sum_{k=1}^{\infty}\gamma_{t}^{(k)}(\omega)\,z^{(k)},\\
				f'(t,\omega,y,z)
				&=
				f^{\prime 1}(t,\omega,y)
				+\sum_{k=1}^{\infty}\gamma_{t}^{\prime(k)}(\omega)\,z^{(k)},
			\end{aligned}
			\right.
		\end{equation}
		for $y\in\mathbb{R}$ and $z=(z^{(k)})_{k\ge1}\in\ell^{2}$, where
		$\big(\gamma^{(k)}\big)_{k\ge1}$ and $\big(\gamma^{\prime(k)}\big)_{k\ge1}$
		are predictable processes satisfying
		\begin{equation}\label{eq:lem23-gamma-L2}
			\mathbb{E}\!\left[\int_{0}^{T}\sum_{k=1}^{\infty}\big|\gamma_{t}^{(k)}\big|^{2}\,dt\right]
			+\mathbb{E}\!\left[\int_{0}^{T}\sum_{k=1}^{\infty}\big|\gamma_{t}^{\prime(k)}\big|^{2}\,dt\right]
			<\infty.
		\end{equation}
		
		Moreover, $f^{1}$ and $f^{\prime 1}$ are progressively measurable and satisfy
		\begin{equation}\label{eq:lem23-integrability}
			\mathbb{E}\!\left[\int_{0}^{T}\big|f^{1}(t,0)\big|^{2}\,dt\right]
			+\mathbb{E}\!\left[\int_{0}^{T}\big|f^{\prime 1}(t,0)\big|^{2}\,dt\right]
			<\infty,
		\end{equation}
		and there exists a constant $\kappa>0$ such that for all $y_{1},y_{2}\in\mathbb{R}$,
		\begin{equation}\label{eq:lem23-lipschitz}
			\big|f^{1}(t,\omega,y_{1})-f^{1}(t,\omega,y_{2})\big|
			\le \kappa |y_{1}-y_{2}|,
			\qquad
			\big|f^{\prime 1}(t,\omega,y_{1})-f^{\prime 1}(t,\omega,y_{2})\big|
			\le \kappa |y_{1}-y_{2}|,
			\quad \mathbb{P}\text{-a.s.}
		\end{equation}
		
		Assume further that for all $(t,\omega)$,
		\begin{equation}\label{eq:lem23-jump-cond}
			\sum_{k=1}^{\infty}\gamma_{t}^{(k)}(\omega)\,\Delta H_{t}^{(k)}(\omega)>-1,
			\qquad
			\sum_{k=1}^{\infty}\gamma_{t}^{\prime(k)}(\omega)\,\Delta H_{t}^{(k)}(\omega)>-1.
		\end{equation}

		Finally, assume the ordering conditions:
		\begin{equation}\label{eq:lem23-order}
			\eta\le \eta',\quad \mathbb{P}\text{-a.s.},
			\qquad
			f^{1}(t,y)\le f^{\prime 1}(t,y),
			\quad (d\mathbb{P}\times dt)\text{-a.e. for all }y\in\mathbb{R}.
		\end{equation}

		Then, for all $t\in[0,T]$,
		\[
		Y_{t}\le Y'_{t},\qquad \mathbb{P}\text{-a.s.}
		\]
	\end{lem}
	
	\begin{rmk}\label{rem:Teugels-comparison}

The condition
\[
\sum_{k\ge1}\gamma_t^{(k)}\Delta H_t^{(k)} > -1
\]
plays a crucial role in the comparison theorem for BSDEs driven by Teugels martingales, as it guarantees the admissibility
of the associated exponential change of measure and ensures the validity of the comparison result.
	\end{rmk}

	\subsection{$L^2$ a priori estimates for SDEs and BSDEs}
	
	\begin{lem}\label{lem:2.4}
		Let the initial value $\xi \in L^2(\Omega,\mathcal{F}_0,\mathbb{P};\mathbb{R}^n)$, and assume that the coefficient function
		$F:[0,T]\times\mathbb{R}^n\to\mathbb{R}^n$
		is Lipschitz continuous with respect to the space variable and satisfies
		\[
		\mathbb E\int_0^T |F(s,0)|^2\,ds < \infty.
		\]
		Consider the stochastic differential equation driven by the L\'{e}vy process $L$,
		\[
		X_t
		=
		\xi
		+
		\int_0^t F(s,X_{s-})\,dL_s,
		\qquad t\in[0,T].
		\]
		Then there exists a constant $C>0$ such that the solution process $X$ satisfies the following a priori estimate:
		\[
		\mathbb{E}\Big[
		\sup_{0\le s\le T}|X_s|^2
		\Big]
		\le
		C\Big(
		\mathbb{E}|\xi|^2
		+
		\mathbb{E}\int_0^T |F(s,0)|^2\,ds
		\Big).
		\]
	\end{lem}
	
	\begin{rmk}
    Lemma~\ref{lem:2.4} follows from the Burkholder-Davis-Gundy (BDG) inequality for L\'{e}vy-driven stochastic integrals. Specifically, for a L\'{e}vy process $L$ with characteristics $(b,\sigma,\nu)$ and a predictable process $\Phi$, Protter and Talay \cite{protter1997euler} established the moment estimate
    \[
    \mathbb{E}\left[ \sup_{0\le t\le T} \left| \int_0^t \Phi_{s}\,dL_s \right|^2 \right] \le C_1 \mathbb{E}\left[ \int_0^T |\Phi_s|^2\,ds \right],
    \]
    where the constant $C_1$ depends on $T$ and $\sigma^2 + \int_{\mathbb{R}} x^2 \nu(dx)$ (see also Protter \cite{protter2005stochastic} for the general theory). Applying this estimate with $\Phi_s = F(s,X_{s-})$ and combining with the Lipschitz condition on $F$ and Gronwall's lemma yields the desired result.
\end{rmk}
	
	\begin{lem}\label{lem:2.5}
		Let $\{H^{(k)}\}_{k\ge1}$ be the family of Teugels martingales associated with the above L\'{e}vy process.
		Given a terminal condition $\eta$ and a generator
		$f:[0,T]\times\Omega\times\mathbb{R}\times\ell^2\to\mathbb{R}$,
		where $\eta$ is $\mathcal{F}_T$-measurable and satisfies
		$\mathbb{E}|\eta|^2<\infty$,
		and $f$ is Lipschitz continuous in $(y,z)$, and
		\[
		\mathbb{E}\int_0^T |f(s,0,0)|^2\,ds<\infty.
		\]
		Consider the backward stochastic differential equation driven by the Teugels martingale system \eqref{eq:2.1}.
		Under the above assumptions, its solution $(Y,Z)$ satisfies the following $L^2$ a priori estimate:
		\[
		\mathbb{E}\Bigg[
		\sup_{0\le s\le T}|Y_s|^2
		+
		\int_0^T \|Z_s\|_{\ell^2}^2\,ds
		\Bigg]
		\le
		C\Bigg(
		\mathbb{E}|\eta|^2
		+
		\mathbb{E}\int_0^T |f(s,0,0)|^2\,ds
		\Bigg).
		\]
	\end{lem}
	
	\begin{rmk}
		The proof of this lemma can be referred to Lemma~2.2 in Maozhong Xu et al. \cite{xu2023class}.
		Both results rely on the orthogonality of Teugels martingales and on the existence, uniqueness, and stability theory of L\'{e}vy-driven BSDEs.
		Since the present paper does not involve the delay and anticipation structures appearing in the original lemma, the corresponding assumptions and estimation procedure are simplified, while the overall proof strategy remains consistent with that of the cited work.
	\end{rmk}

	\section{DPP for Stochastic Optimal Control of FBSDEs with L\'{e}vy Jumps}\label{sec:3}
	\subsection{Problem formulation and well-posedness of solutions}
	
	For $t\in[0, T)$, denote by $\mathcal{U}[t, T]$ the set of all $\mathbb{F}$-adapted $U$-valued processes on $[t,T]$.
	
	In order to establish the dynamic programming principle for stochastic optimal control problems driven by L\'{e}vy processes, we first specify the framework of the controlled system studied in this paper.
	Given an initial time $t \in [0,T]$, an initial state $\xi \in L^2(\Omega, \mathcal{F}_t, \mathbb{P};\mathbb{R}^n)$, and a control process $u(\cdot) \in \mathcal{U}[t,T]$, consider the following controlled decoupled FBSDE:
	\begin{equation}\label{eq:3.1}
		\begin{cases}
			dX_s^{t,\xi;u} = F(s, X_{s-}^{t,\xi;u}, u_s) \, dL_s,\\[2mm]
			dY_s^{t,\xi;u} = -f(s, X_{s-}^{t,\xi;u}, Y_s^{t,\xi;u}, Z_s^{t,\xi;u}, u_s) \, ds
			+ \sum_{k=1}^{\infty} Z_s^{t,\xi;u,(k)} \, dH_s^{(k)},\\[1mm]
			X_t^{t,\xi;u} = \xi, \quad Y_T^{t,\xi;u} = \phi(X_T^{t,\xi;u}).
		\end{cases}
	\end{equation}
	where
	\[
	F : [0,T] \times \mathbb{R}^n \times U \;\longrightarrow\; \mathbb{R}^n,
	\quad
	f : [0,T] \times \mathbb{R}^n \times \mathbb{R} \times \ell^2 \times U
	\;\longrightarrow\; \mathbb{R}.
	\]
	In the above system, the control objective is to minimize the cost at the initial time $Y_t$.
	Therefore, for any $(t,x) \in [0,T] \times \mathbb{R}^n$, the value function is defined by
	\begin{equation}\label{eq:3.2}
		W(t,x) := \operatorname*{ess\,inf}_{u \in \mathcal{U}[t,T]} Y_t^{t,x;u}.
	\end{equation}
	It is well known that $W(t,x)$ is deterministic.
	
	\begin{ass}\label{ass:3.1}
		Let $F, f, \phi$ be the coefficient functions of the L\'{e}vy-driven FBSDE. Assume that $F$ is continuous with respect to $(s,x,u)$ and $f$ is continuous with respect to $(s,x,y,z,u)$, and that they satisfy the following Lipschitz
		conditions: there exist constants $L_1, L_2, L_3 > 0$ such that for any $s \in [0,T]$, $x_1, x_2 \in \mathbb{R}^n$, $y_1, y_2 \in \mathbb{R}$, $z_1, z_2 \in \ell^2$, and $u \in U$, it holds that
		\[
		|F(s,x_1,u)-F(s,x_2,u)| \le L_1 |x_1-x_2|,
		\]
		\[
		|f(s,x_1,y_1,z_1,u)-f(s,x_2,y_2,z_2,u)| \le L_1 |x_1-x_2| + L_2 |y_1-y_2| + L_3 \|z_1-z_2\|_{\ell^2},
		\]
		\[
		|\phi(x_1)-\phi(x_2)| \le L_1 |x_1-x_2|.
		\]
		Moreover, the generator $f$ admits the following decomposition with respect to $z$:
		\[
		f(t,\omega,y,z,u)
		=
		f^1(t,\omega,y,u)
		+
		\sum_{k=1}^\infty \gamma_t^{(k)}(\omega,u)\, z^{(k)},
		\]
		where $(\gamma^{(k)})_{k\ge1}$ are predictable processes with respect to $\mathbb{F}$, depending on the control $u$, and satisfying
		\[
		\mathbb{E}\!\left[
		\int_0^T \sum_{k=1}^\infty |\gamma_t^{(k)}(\omega,u)|^2\,dt
		\right]
		<\infty,
		\]
		and almost surely for all $t\in[0,T]$ and $u\in U$,
		\[
		\sum_{k=1}^\infty \gamma_t^{(k)}(\omega,u)\,\Delta H_t^{(k)}(\omega) > -1.
		\]

	\end{ass}
	
	\begin{ass}\label{ass:3.2}
		Under the Lipschitz conditions and the compactness of $U$, the coefficients satisfy the following linear growth condition: there exists a constant $L>0$ such that for any $s \in [0,T]$, $x \in \mathbb{R}^n$, $y \in \mathbb{R}$, $z \in \ell^2$, and $u \in U$, one has
		\[
		|F(s,x,u)| + |f(s,x,y,z,u)| + |\phi(x)| \le L(1 + |x| + |y| + \|z\|_{\ell^2}).
		\]
		This property will be used in the subsequent moment estimates and stability analysis of the solutions.
	\end{ass}
	
	Under Assumptions~\ref{ass:3.1} and~\ref{ass:3.2}, by applying Lemma~\ref{lem:2.4} to the forward equation and Lemma~\ref{lem:2.5}  to the backward equation, it follows that for any $t \in [0,T]$, $\xi \in L^2(\Omega,\mathcal{F}_t;\mathbb{R}^n)$, and $u \in \mathcal{U}[t,T]$, the controlled FBSDE \eqref{eq:3.1} admits a unique adapted solution
	\[
	\bigl(
	X^{t,\xi;u},
	Y^{t,\xi;u},
	Z^{t,\xi;u}
	\bigr)
	\in
	\mathcal{S}^2(t,T;\mathbb{R}^n)
	\times
	\mathcal{S}^2(t,T;\mathbb{R})
	\times
	\mathcal{P}^2(t,T;\ell^2).
	\]
	
	Based on the above well-posedness of the solution, the dynamic programming principle for the value function $W(t,x)$ can then be established.

\subsection{The Dynamic Programming Principle}

To connect the forward-backward stochastic system with the Hamilton-Jacobi-Bellman (HJB) equation,
we first establish the dynamic programming principle (DPP). Serving as the key link between the FBSDE formulation
and the viscosity solution framework, the DPP ensures that the value function is Lipschitz continuous
and can be rigorously characterized as a viscosity solution.

As a fundamental tool in stochastic optimal control theory (see \cite{yong1999stochastic, zhou1991unified, zhou1990connection}),
the DPP allows the reduction of complex control problems to the analysis of the associated value function.
In this subsection, we formulate the DPP for the control problem defined in \eqref{eq:3.1}--\eqref{eq:3.2},
laying the foundation for subsequent viscosity solution analysis.
	
	Define $\mathcal{U}^{t}[t, T]$ as the set of all $U$-valued processes on $[t, T]$ that are adapted to the filtration
	$\{\mathcal{F}_s^t\}_{t \le s \le T}$, where $\{\mathcal{F}_s^t\}_{t \le s \le T}$ denotes the $\mathbb{P}$-augmented natural filtration generated by the increment process $(L_s - L_t)_{t \le s \le T}$. For any $v \in \mathcal{U}^{t}[t, T]$, under Assumption~\ref{ass:3.1} and the well-posedness theory of L\'{e}vy-driven BSDEs, the controlled decoupled FBSDE \eqref{eq:3.1} admits a unique solution
	$(X_s^{t,\xi;v}, Y_s^{t,\xi;v}, Z_s^{t,\xi;v})_{s \in [t,T]}$
	which is adapted to $\{\mathcal{F}_s^t\}_{t \le s \le T}$. In particular, $Y_t^{t,\xi;v}$ is $\mathcal{F}_t^t$-measurable. Since the initial $\sigma$-algebra $\mathcal{F}_t^t$ is trivial (i.e., $\mathcal{F}_t^t = \{\emptyset, \Omega\}$), it follows that $Y_t^{t,\xi;v}$ is deterministic.
	
	It should be noted that in the subsequent proofs of this paper, the constant $C$ may vary from line to line.
	
	\begin{prop}\label{prop:3.1}
		Under Assumption \ref{ass:3.1}, we have
		\begin{equation}\label{eq:3.3}
			W(t,x)=\inf_{v\in\mathcal{U}^{t}[t,T]}Y_t^{t,x;v}.
		\end{equation}
	\end{prop}
	
	\begin{proof}
		Since $\mathcal{U}^{t}[t,T]\subset \mathcal{U}[t,T]$, by the definition of the value function $W(t,x)$, we directly have
		\[
		W(t,x)\le \inf_{v\in\mathcal{U}^{t}[t,T]}Y_t^{t,x;v}.
		\]
		
		We now prove the converse inequality. Fix any $u\in \mathcal{U}[t,T]$. By Lemma 13 in \cite{huji2017dynamic}, there exists a sequence $(u^m)_{m\ge1}\subset \mathcal{U}[t,T]$ such that
		\[
		\mathbb{E}\int_t^T |u_s^m-u_s|^2\,ds\to 0,\qquad m\to\infty.
		\]
		Moreover, each $u^m$ can be represented in the form
		\[
		u_s^m=\sum_{i=1}^m v_s^{i,m}\mathbf{1}_{A_i},\qquad s\in[t,T],
		\]
		where $\{A_i\}_{i=1}^m$ is a partition of $(\Omega,\mathcal{F}_t)$ and each $v^{i,m}\in \mathcal{U}^{t}[t,T]$.
		
		For each fixed $m$, define
		\[
		\widetilde X_s^m:=\sum_{i=1}^m X_s^{t,x;v^{i,m}}\mathbf{1}_{A_i},\qquad
		\widetilde Y_s^m:=\sum_{i=1}^m Y_s^{t,x;v^{i,m}}\mathbf{1}_{A_i},\qquad
		\widetilde Z_s^m:=\sum_{i=1}^m Z_s^{t,x;v^{i,m}}\mathbf{1}_{A_i},\qquad s\in[t,T].
		\]
		Since $u^m=\sum_{i=1}^m v^{i,m}\mathbf{1}_{A_i}$ and $\{A_i\}_{i=1}^m\subset \mathcal{F}_t$, it is straightforward to verify that
		\[
		(\widetilde X^m,\widetilde Y^m,\widetilde Z^m)
		\]
		satisfies the same controlled FBSDE \eqref{eq:3.1} as
		\[
		(X^{t,x;u^m},Y^{t,x;u^m},Z^{t,x;u^m})
		\]
		with control $u^m$. By the uniqueness of the solution to \eqref{eq:3.1}, it follows that
		\[
		(X_s^{t,x;u^m},Y_s^{t,x;u^m},Z_s^{t,x;u^m})_{s\in[t,T]}
		=
		\Big(
		\sum_{i=1}^m X_s^{t,x;v^{i,m}}\mathbf{1}_{A_i},
		\sum_{i=1}^m Y_s^{t,x;v^{i,m}}\mathbf{1}_{A_i},
		\sum_{i=1}^m Z_s^{t,x;v^{i,m}}\mathbf{1}_{A_i}
		\Big)_{s\in[t,T]}.
		\]
		
		Next, we establish the convergence of $Y_t^{t,x;u^m}$ to $Y_t^{t,x;u}$. Set
		\[
		\widehat X_s:=X_s^{t,x;u^m}-X_s^{t,x;u},\qquad
		\widehat Y_s:=Y_s^{t,x;u^m}-Y_s^{t,x;u},\qquad
		\widehat Z_s:=Z_s^{t,x;u^m}-Z_s^{t,x;u},\qquad s\in[t,T].
		\]
		Then $(\widehat X,\widehat Y,\widehat Z)$ satisfies
		\[
		\left\{
		\begin{aligned}
			d\widehat X_s
			&=
			\Big(F(s,X_{s-}^{t,x;u^m},u_s^m)-F(s,X_{s-}^{t,x;u},u_s)\Big)\,dL_s,\\
			d\widehat Y_s
			&=
			-\Big(f(s,X_{s-}^{t,x;u^m},Y_s^{t,x;u^m},Z_s^{t,x;u^m},u_s^m)
			-f(s,X_{s-}^{t,x;u},Y_s^{t,x;u},Z_s^{t,x;u},u_s)\Big)\,ds\\
			&\quad
			+\sum_{k\ge1}\widehat Z_s^{(k)}\,dH_s^{(k)},\\
			\widehat X_t&=0,\qquad
			\widehat Y_T=\phi(X_T^{t,x;u^m})-\phi(X_T^{t,x;u}).
		\end{aligned}
		\right.
		\]
		By the Lipschitz continuity of $F$, $f$, and $\phi$, we have
		\[
		|F(s,X_{s-}^{t,x;u^m},u_s^m)-F(s,X_{s-}^{t,x;u},u_s)|
		\le C\big(|\widehat X_{s-}|+|u_s^m-u_s|\big),
		\]
		\[
		\begin{aligned}
			&|f(s,X_{s-}^{t,x;u^m},Y_s^{t,x;u^m},Z_s^{t,x;u^m},u_s^m)
			-f(s,X_{s-}^{t,x;u},Y_s^{t,x;u},Z_s^{t,x;u},u_s)|\\
			&\qquad \le
			C\big(|\widehat X_s|+|\widehat Y_s|+\|\widehat Z_s\|_{\ell^2}+|u_s^m-u_s|\big),
		\end{aligned}
		\]
		and
		\[
		|\phi(X_T^{t,x;u^m})-\phi(X_T^{t,x;u})|
		\le C|\widehat X_T|.
		\]
		
		Applying Lemma \ref{lem:2.4} to the forward difference equation yields
		\begin{equation}\label{eq:3.4}
			\mathbb{E}\Big[\sup_{t\le s\le T}|\widehat X_s|^2\Big]
			\le
			C\,\mathbb{E}\int_t^T |u_s^m-u_s|^2\,ds.
		\end{equation}
		Then, by Lemma \ref{lem:2.5} applied to the backward difference equation, together with \eqref{eq:3.4}, we obtain
		\begin{equation}\label{eq:3.5}
			\mathbb{E}\Big[\sup_{t\le s\le T}|\widehat Y_s|^2+\int_t^T\|\widehat Z_s\|_{\ell^2}^2\,ds\Big]
			\le
			C\,\mathbb{E}\Big[|\widehat X_T|^2+\int_t^T\big(|\widehat X_s|^2+|u_s^m-u_s|^2\big)\,ds\Big]
			\le
			C\,\mathbb{E}\int_t^T |u_s^m-u_s|^2\,ds.
		\end{equation}
		Consequently,
		\[
		\mathbb{E}\big[|Y_t^{t,x;u^m}-Y_t^{t,x;u}|^2\big]
		\to 0,\qquad m\to\infty.
		\]
		
		On the other hand, by the above decomposition,
		\[
		Y_t^{t,x;u^m}
		=
		\sum_{i=1}^m Y_t^{t,x;v^{i,m}}\mathbf{1}_{A_i}
		\ge
		\inf_{v\in\mathcal{U}^{t}[t,T]}Y_t^{t,x;v},
		\qquad \mathbb{P}\text{-a.s.}
		\]
		Passing to the limit as $m\to\infty$, we get
		\[
		Y_t^{t,x;u}\ge \inf_{v\in\mathcal{U}^{t}[t,T]}Y_t^{t,x;v},
		\qquad \mathbb{P}\text{-a.s.}
		\]
		Since $u\in\mathcal{U}[t,T]$ is arbitrary, taking the infimum over $u\in\mathcal{U}[t,T]$ yields
		\[
		W(t,x)\ge \inf_{v\in\mathcal{U}^{t}[t,T]}Y_t^{t,x;v}.
		\]
		Combining this with the opposite inequality proved at the beginning, we conclude that
		\[
		W(t,x)=\inf_{v\in\mathcal{U}^{t}[t,T]}Y_t^{t,x;v}.
		\]
	\end{proof}

	\begin{lem}\label{lem:3.2}
		Under Assumption \ref{ass:3.1}, there exists a constant $C>0$, depending only on
		$L_1, L_2, L_3$, $T$, and
		\[
		\int_{\mathbb{R}} (1 \wedge |x|^2)\, \nu(dx),
		\]
		such that for any $u \in \mathcal{U}[t,T]$ and any
		$\xi,\xi' \in L^2(\mathcal{F}_t;\mathbb{R}^n)$, the solutions of the FBSDE \eqref{eq:3.1} satisfy
		\begin{equation}\label{eq:3.7}
			\mathbb{E}\Big[
			\sup_{s\in[t,T]}|X_s^{t,\xi;u}-X_s^{t,\xi';u}|^2
			+
			\sup_{s\in[t,T]}|Y_s^{t,\xi;u}-Y_s^{t,\xi';u}|^2
			+
			\int_t^T \|Z_s^{t,\xi;u}-Z_s^{t,\xi';u}\|_{\ell^2}^2 ds
			\Big] \le C |\xi-\xi'|^2.
		\end{equation}
		
		\begin{equation} \label{eq:3.8}
			\mathbb{E}\Biggl[ \sup_{s\in[t,T]} \Bigl( |X_s^{t,\xi;u}|^2 + |Y_s^{t,\xi;u}|^2 \Bigr)
			+ \int_t^T \|Z_s^{t,\xi;u}\|_{\ell^2}^2 ds \Biggr]
			\leq C \, \bigl( 1 + |\xi|^2 \bigr).
		\end{equation}
	\end{lem}
	
	\begin{proof}
		For simplicity of notation, we present the proof in the one-dimensional case; the multidimensional case follows similarly.
		
		Set
		\[
		\hat X_s := X_s^{t,\xi;u} - X_s^{t,\xi';u}, \quad
		\hat Y_s := Y_s^{t,\xi;u} - Y_s^{t,\xi';u}, \quad
		\hat Z_s := Z_s^{t,\xi;u} - Z_s^{t,\xi';u}.
		\]
		
		By the forward equation,
		\[
		\hat X_s = \xi - \xi' + \int_t^s \big[F(r,X_{r-}^{t,\xi;u},u_r) - F(r,X_{r-}^{t,\xi';u},u_r)\big] dL_r.
		\]
		
		Applying Lemma \ref{lem:2.4} to the above SDE yields
		\begin{equation}\label{eq:3.9}
			\mathbb{E}\Big[\sup_{s\in[t,T]} |\hat X_s|^2\Big] \le C |\xi - \xi'|^2.
		\end{equation}
		
		Next, for the backward equation, define the generator difference
		\[
		\hat f(s) := f(s,X_{s-}^{t,\xi;u},Y_s^{t,\xi;u},Z_s^{t,\xi;u},u_s)
		- f(s,X_{s-}^{t,\xi';u},Y_s^{t,\xi';u},Z_s^{t,\xi';u},u_s).
		\]
		
		By Assumption \ref{ass:3.1}, $\hat f$ is Lipschitz continuous with respect to $(\hat Y_s, \hat Z_s)$:
		\[
		|\hat f(s)| \le L_1 |\hat X_{s-}| + L_2 |\hat Y_s| + L_3 \|\hat Z_s\|_{\ell^2}.
		\]
		
		Then the difference $\hat Y_s$ satisfies
		\[
		\hat Y_s = \phi(X_T^{t,\xi;u}) - \phi(X_T^{t,\xi';u})
		+ \int_s^T \hat f(r) dr
		- \sum_{k=1}^\infty \int_s^T \hat Z_r^{(k)} dH_r^{(k)}, \quad s\in[t,T].
		\]
		
		Applying Lemma \ref{lem:2.5} to this BSDE, together with the estimate \eqref{eq:3.9} and the Lipschitz property of $\phi$, we obtain
		\begin{align*}
			\mathbb{E}\Big[\sup_{s\in[t,T]} |\hat Y_s|^2 + \int_t^T \|\hat Z_s\|_{\ell^2}^2 ds \Big]
			&\le C \mathbb{E}|\phi(X_T^{t,\xi;u}) - \phi(X_T^{t,\xi';u})|^2
			+ C \int_t^T \mathbb{E}|\hat X_{s-}|^2 ds \\
			&\le C \mathbb{E}\Big[\sup_{s\in[t,T]} |\hat X_s|^2 \Big] \le C |\xi-\xi'|^2.
		\end{align*}
		
		Combining the forward and backward estimates, we conclude the desired inequality\eqref{eq:3.7}.
		
		Next, we prove the estimate \eqref{eq:3.8}.
		
		For the forward equation, applying Lemma \ref{lem:2.4} to the SDE, together with the Lipschitz continuity and linear growth of $F$, we obtain
		\begin{equation}\label{eq:3.10}
			\mathbb{E}\Big[\sup_{s\in[t,T]} |X_s^{t,\xi;u}|^2 \Big]
			\le C \bigl(1 + |\xi|^2\bigr).
		\end{equation}
		
		Next, we consider the backward equation. By the Lipschitz property of $f$ and the linear growth condition, we have
		\[
		|f(s,X_{s-}^{t,\xi;u},Y_s^{t,\xi;u},Z_s^{t,\xi;u},u_s)|
		\le C \bigl(1 + |X_{s-}^{t,\xi;u}| + |Y_s^{t,\xi;u}| + \|Z_s^{t,\xi;u}\|_{\ell^2}\bigr).
		\]
		
		Applying Lemma \ref{lem:2.5} to the BSDE, we obtain
		\begin{align*}
			\mathbb{E}\Big[\sup_{s\in[t,T]} |Y_s^{t,\xi;u}|^2
			+ \int_t^T \|Z_s^{t,\xi;u}\|_{\ell^2}^2 ds \Big]
			&\le C \mathbb{E}|\phi(X_T^{t,\xi;u})|^2
			+ C \mathbb{E}\int_t^T \bigl(1 + |X_s^{t,\xi;u}|^2\bigr) ds.
		\end{align*}
		
		Using the Lipschitz continuity of $\phi$, we have
		\[
		|\phi(X_T^{t,\xi;u})|^2 \le C\bigl(1 + |X_T^{t,\xi;u}|^2\bigr).
		\]
		
		Combining this with \eqref{eq:3.10}, we deduce
		\begin{align*}
			\mathbb{E}\Big[\sup_{s\in[t,T]} |Y_s^{t,\xi;u}|^2
			+ \int_t^T \|Z_s^{t,\xi;u}\|_{\ell^2}^2 ds \Big]
			&\le C \bigl(1 + \mathbb{E}\sup_{s\in[t,T]} |X_s^{t,\xi;u}|^2 \bigr) \\
			&\le C \bigl(1 + |\xi|^2\bigr).
		\end{align*}
		
		Combining the estimates for $X$, $Y$, and $Z$, we obtain \eqref{eq:3.8}.
	\end{proof}
	
	\begin{lem}\label{lem:3.3}
		Under Assumption \ref{ass:3.1}, there exist two constants $C$ and $C'$, depending only on $L_1, L_2, L_3$, $T$, and
		\[
		\int_{\mathbb{R}} (1 \wedge |x|^2)\, \nu(dx),
		\]
		such that for any $t \in [0,T]$ and any $x, x' \in \mathbb{R}^n$, we have
		\[
		|W(t,x) - W(t,x')| \le C |x - x'|,
		\quad \text{and} \quad
		|W(t,x)| \le C' (1 + |x|).
		\]
	\end{lem}
	
	\begin{proof}
		By Proposition~\ref{prop:3.1} and Lemma~\ref{lem:3.2}, we obtain
		\[
		\begin{aligned}
			|W(t,x) - W(t,x')|
			&= \Biggl|
			\inf_{v \in \mathcal{U}^{t}[t,T]} Y_t^{t,x;v}
			- \inf_{v \in \mathcal{U}^{t}[t,T]} Y_t^{t,x';v}
			\Biggr| \\[1mm]
			&\le \sup_{v \in \mathcal{U}^{t}[t,T]}
			\bigl| Y_t^{t,x;v} - Y_t^{t,x';v} \bigr| \\[1mm]
			&\le \sup_{v \in \mathcal{U}^{t}[t,T]}
			\Biggl(
			\mathbb{E} \Big[
			\sup_{t \le s \le T} | Y_s^{t,x;v} - Y_s^{t,x';v} |^2
			\Big]
			\Biggr)^{1/2} \\[1mm]
			&\le C |x - x'|.
		\end{aligned}
		\]
		For the linear growth estimate, by Proposition~\ref{prop:3.1} and the estimate \eqref{eq:3.8} in Lemma~\ref{lem:3.2}, we have
		\[
		|W(t,x)|
		\le \sup_{v \in \mathcal{U}^{t}[t,T]} |Y_t^{t,x;v}|
		\le C'(1+|x|).
		\]
	\end{proof}

	Before studying the dynamic programming principle (DPP), we introduce the notion of the backward semigroup. This concept was first proposed by Peng in \cite{peng1997backward}. For given
	\((t, x) \in [0, T) \times \mathbb{R}^n\), \(\delta \in (0, T - t]\) and \(u \in \mathcal{U}[t, t + \delta]\),
	for each \(\psi \in \mathcal{L}(\mathbb{R}^n)\), define
	\[
	G_{t,t+\delta}^{t,x;u}\left[ \psi\left(\widetilde{X}_{t+\delta}^{t,x;u}\right) \right]
	= \widetilde{Y}_t^{t,x;u},
	\]
	where
	\[
	\mathcal{L}(\mathbb{R}^n)
	= \Bigl\{ \psi: \mathbb{R}^n \to \mathbb{R} \;\big|\; \psi \text{ is Lipschitz continuous} \Bigr\}.
	\]
	
	Moreover, \((\widetilde{X}^{t,x;u}, \widetilde{Y}^{t,x;u}, \widetilde{Z}^{t,x;u})\) is the solution on the interval \([t, t + \delta]\) of the following forward-backward stochastic differential equation (FBSDE):
	\begin{equation}\label{eq:3.11}
		\begin{cases}
			d\widetilde{X}_s^{t,x;u} = F\bigl(s, \widetilde{X}_{s-}^{t,x;u}, u_s\bigr)\, dL_s, \\[0.3em]
			d\widetilde{Y}_s^{t,x;u} = -f\bigl(s, \widetilde{X}_{s-}^{t,x;u}, \widetilde{Y}_{s}^{t,x;u}, \widetilde{Z}_s^{t,x;u}, u_s\bigr)\, ds
			+ \displaystyle\sum_{k=1}^{\infty} \widetilde{Z}_s^{t,x;u,(k)} \, dH_s^{(k)},
			\quad s \in [t, t+\delta], \\[0.6em]
			\widetilde{X}_t^{t,x;u} = x,
			\quad \widetilde{Y}_{t+\delta}^{t,x;u} = \psi\bigl(\widetilde{X}_{t+\delta}^{t,x;u}\bigr).
		\end{cases}
	\end{equation}
	
	Since the coefficients of equation \eqref{eq:3.11} satisfy Assumption~\ref{ass:3.1}, the above equation admits a unique solution \\
	\((\widetilde{X}^{t,x;u}, \widetilde{Y}^{t,x;u}, \widetilde{Z}^{t,x;u})\). Hence, the mapping
	\(G_{t,t+\delta}^{t,x;u}[\cdot]\) is well defined.
	
	\begin{thm}\label{thm:3.5}
		Under Assumptions~\ref{ass:3.1}--\ref{ass:3.2}, for any $(t,x)\in[0,T)\times\mathbb{R}^n$ and any $\delta\in(0,T-t]$, the value function satisfies
		\begin{equation}\label{eq:3.17}
			W(t,x)
			=
			\inf_{v\in\mathcal{U}^t[t,t+\delta]}
			G_{t,t+\delta}^{t,x;v}
			\Big[
			W\bigl(t+\delta,\widetilde X_{t+\delta}^{t,x;v}\bigr)
			\Big].
		\end{equation}
	\end{thm}
	
	\begin{proof}
		We divide the proof into two steps.
		
		\medskip
		\noindent\emph{Step 1: Proof of ``$\ge$''.}
		
		For any $v(\cdot)\in\mathcal{U}^t[t,T]$, by the flow property of FBSDEs and the definition of the backward semigroup, we have
		\[
		Y_t^{t,x;v}
		=
		G_{t,t+\delta}^{t,x;v}
		\Big[
		Y_{t+\delta}^{t,x;v}
		\Big].
		\]
		By the definition of the value function,
		\[
		W\bigl(t+\delta,\widetilde X_{t+\delta}^{t,x;v}\bigr)
		\le
		Y_{t+\delta}^{t,x;v},
		\quad \mathbb P\text{-a.s.}
		\]
		Applying the comparison theorem (Lemma~\ref{lem:2.3}), we obtain
		\[
		G_{t,t+\delta}^{t,x;v}
		\Big[
		W\bigl(t+\delta,\widetilde X_{t+\delta}^{t,x;v}\bigr)
		\Big]
		\le
		Y_t^{t,x;v}.
		\]
		Taking the infimum over $v\in\mathcal{U}^t[t,T]$, we get
		\[
		\inf_{v\in\mathcal{U}^t[t,T]}
		G_{t,t+\delta}^{t,x;v}
		\Big[
		W\bigl(t+\delta,\widetilde X_{t+\delta}^{t,x;v}\bigr)
		\Big]
		\le
		W(t,x).
		\]
		
		Since the semigroup depends only on the restriction of $v$ on $[t,t+\delta]$, and any control in $\mathcal{U}^t[t,t+\delta]$ can be extended arbitrarily to an element of $\mathcal{U}^t[t,T]$, the above infimum is equivalent to the infimum over $\mathcal{U}^t[t,t+\delta]$. Hence,
		\[
		\inf_{v\in\mathcal{U}^t[t,t+\delta]}
		G_{t,t+\delta}^{t,x;v}
		\Big[
		W\bigl(t+\delta,\widetilde X_{t+\delta}^{t,x;v}\bigr)
		\Big]
		\le
		W(t,x).
		\]
		
		\medskip
		\noindent\emph{Step 2: Proof of ``$\le$''.}
		
		Fix any $u(\cdot)\in\mathcal{U}[t,t+\delta]$. Let $(\widetilde X,\widetilde Y,\widetilde Z)$ be the solution on $[t,t+\delta]$ of
		\begin{equation}\label{eq:3.18}
			\begin{cases}
				d\widetilde X_s = F(s,\widetilde X_{s-},u_s)\,dL_s,\\
				d\widetilde Y_s = -f(s,\widetilde X_{s-},\widetilde Y_s,\widetilde Z_s,u_s)\,ds
				+ \sum_{k=1}^\infty \widetilde Z_s^{(k)}\,dH_s^{(k)},\\
				\widetilde X_t = x,\quad
				\widetilde Y_{t+\delta}
				=
				W\bigl(t+\delta,\widetilde X_{t+\delta}\bigr).
			\end{cases}
		\end{equation}
		Then
		\[
		\widetilde Y_t
		=
		G_{t,t+\delta}^{t,x;u}
		\Big[
		W\bigl(t+\delta,\widetilde X_{t+\delta}^{t,x;u}\bigr)
		\Big].
		\]
		
		For each $m\ge1$, let $\xi^m$ be a simple approximation of $\widetilde X_{t+\delta}$ such that
		\[
		\mathbb E\bigl[|\xi^m-\widetilde X_{t+\delta}|^2\bigr]\to0.
		\]
		For each $x_i^m$, choose $v^{i,m}(\cdot)\in\mathcal{U}^t[t+\delta,T]$ satisfying
		\begin{equation}\label{eq:3.19}
			W(t+\delta,x_i^m)
			\le
			Y_{t+\delta}^{t+\delta,x_i^m;v^{i,m}}
			\le
			W(t+\delta,x_i^m)+\frac1m.
		\end{equation}
		Define
		\[
		v^m_s = \sum_i v_s^{i,m}\mathbf1_{A_i^m},
		\]
		and construct the concatenated control
		\[
		u^m_s =
		\begin{cases}
			u_s, & s\in[t,t+\delta],\\
			v^m_s, & s\in(t+\delta,T].
		\end{cases}
		\]
		
		Let $(X^m,Y^m,Z^m)$ be the solution under $u^m$. Then
		\[
		Y_t^{t,x;u^m} = Y_t^m.
		\]
		
		Moreover,
		\[
		W(t+\delta,\xi^m)
		\le
		Y_{t+\delta}^{t+\delta,\xi^m;v^m}
		\le
		W(t+\delta,\xi^m)+\frac1m.
		\]
		
		Since $W(t+\delta,\cdot)$ is Lipschitz continuous (Lemma~\ref{lem:3.3}), it follows that
		\[
		\mathbb E\Big|
		Y_{t+\delta}^{t+\delta,\xi^m;v^m}
		-
		W(t+\delta,\widetilde X_{t+\delta})
		\Big|^2
		\to0.
		\]
		
		Let $(\bar{X}^m, \bar{Y}^m, \bar{Z}^m)$ be the solution to the decoupled FBSDE \eqref{eq:3.18} on $[t, t+\delta]$ with terminal condition $Y_{t+\delta}^{t+\delta, \xi^m; v^m}$. By the standard stability estimate for decoupled FBSDEs with respect to terminal conditions, we have
\[
\mathbb{E}\left[|\bar{Y}_t^m - \widetilde{Y}_t|^2\right] \to 0, \quad \text{as } m\to\infty,
\]
where $\widetilde{Y}_t$ denotes the limit of $\bar{Y}_t^m$ in $L^2(\Omega, \mathcal{F}_t, \mathbb{P})$, corresponding to the control $u \in \mathcal{U}[t, t+\delta]$.

Next, by the stability estimate for decoupled FBSDEs with respect to controls, we obtain
\[
\mathbb{E}\left[|Y_t^{t,x;u^m} - \bar{Y}_t^m|^2\right] \to 0, \quad \text{as } m\to\infty.
\]
Combining the above two convergence results, we conclude that
\[
Y_t^{t,x;u^m} \to \widetilde{Y}_t \quad \text{in } L^2(\Omega, \mathcal{F}_t, \mathbb{P}), \quad \text{as } m\to\infty.
\]

Since $Y_t^{t,x;u^m} \ge W(t,x)$ for all $m \ge 1$, by the lower semicontinuity of the $L^2$ limit (or Fatou's lemma), we have
\[
\widetilde{Y}_t \ge W(t,x).
\]
Recall that $\widetilde{Y}_t = G_{t,t+\delta}^{t,x;u}\left[W(t+\delta, \widetilde{X}_{t+\delta}^{t,x;u})\right]$, where $G_{t,t+\delta}^{t,x;u}[\cdot]$ denotes the solution operator of the FBSDE on $[t, t+\delta]$ with control $u$ and terminal condition $W(t+\delta, \cdot)$. Thus,
\[
G_{t,t+\delta}^{t,x;u}\left[W(t+\delta, \widetilde{X}_{t+\delta}^{t,x;u})\right] \ge W(t,x).
\]

Taking the infimum over all $u \in \mathcal{U}[t, t+\delta]$ on both sides, we obtain
\[
W(t,x) \le \inf_{u\in \mathcal{U}[t, t+\delta]} G_{t,t+\delta}^{t,x;u}\left[W(t+\delta, \widetilde{X}_{t+\delta}^{t,x;u})\right].
\]

Finally, using the same approximation argument as in Proposition 3.1, the infimum over $\mathcal{U}[t, t+\delta]$ can be restricted to $\mathcal{U}^t[t, t+\delta]$. This completes the proof.
\qed
	\end{proof}
	
	\begin{lem}\label{lem:3.5}
		Under Assumption \ref{ass:3.1}, the value function $W(t,x)$ is $1/2$-H\"{o}lder continuous with respect to the time variable $t$.
		That is, there exists a constant $C>0$ such that for any $(t,x) \in [0,T] \times \mathbb{R}^n$ and any $t_0 \in [0,T]$, we have
		\begin{equation}\label{eq:3.19}
			|W(t,x) - W(t_0,x)| \;\leq\; C(1 + |x|) \cdot |t - t_0|^{\frac{1}{2}}.
		\end{equation}
	\end{lem}
	
	\begin{proof}
		It suffices to consider the case of a small time increment.
		Let $\delta>0$ with $t+\delta \leq T$.
		We first prove that
		\begin{equation*}
			|W(t,x) - W(t+\delta,x)| \;\leq\; C(1 + |x|) \delta^{\frac{1}{2}}.
		\end{equation*}
		Then, by taking $\delta = |t - t_0|$, the result can be extended to the general case for any $t,t_0 \in [0,T]$.
		
		By the dynamic programming principle (DPP), the value function satisfies
		\begin{equation*}
			W(t,x) = \essinf_{u \in \mathcal{U}[t,t+\delta]}
			G_{t,t+\delta}^{t,x;u} \Big[ W(t+\delta, \widetilde{X}_{t+\delta}^{t,x;u}) \Big].
		\end{equation*}
		where $G_{t,t+\delta}^{t,x;u}[\cdot]$ denotes the backward semigroup on the interval $[t,t+\delta]$, and $\widetilde{X}_s^{t,x;u}$ is the state process under the control $u$.
		
		For any $\varepsilon > 0$, there exists a control
		$u^* \in \mathcal{U}[t,t+\delta]$ such that
		\begin{equation}\label{eq:3.20}
			W(t,x) \;\leq\;
			G_{t,t+\delta}^{t,x;u^*}
			\Big[ W(t+\delta, \widetilde{X}_{t+\delta}^{t,x;u^*}) \Big]
			\;\leq\; W(t,x) + \varepsilon.
		\end{equation}
		Hence,
		\begin{equation}\label{eq:3.21}
			\begin{aligned}
				W(t,x) - W(t+\delta,x)
				&\leq
				G_{t,t+\delta}^{\,t,x;u^*}
				\Big[ W(t+\delta, \widetilde{X}_{t+\delta}^{t,x;u^*}) \Big]
				- W(t+\delta,x) + \varepsilon .
			\end{aligned}
		\end{equation}
		
		We define two perturbation terms:
		\begin{align*}
			I_\delta^1(u^*) &:=
			G_{t,t+\delta}^{t,x;u^*}
			\Big[ W(t+\delta, \widetilde{X}_{t+\delta}^{t,x;u^*}) \Big]
			- G_{t,t+\delta}^{t,x;u^*}[ W(t+\delta,x) ], \\
			I_\delta^2(u^*) &:=
			G_{t,t+\delta}^{t,x;u^*}[ W(t+\delta,x) ]
			- W(t+\delta,x).
		\end{align*}
		It then follows from \eqref{eq:3.21} that
		\begin{equation}\label{eq:3.22}
			W(t,x) - W(t+\delta,x) \;\leq\; I_\delta^1(u^*) + I_\delta^2(u^*) + \varepsilon.
		\end{equation}
		Since $\varepsilon>0$ is arbitrary, letting $\varepsilon \to 0$ yields
		\begin{equation} \label{eq:3.23}
			W(t,x) - W(t+\delta,x) \;\leq\; I_\delta^1(u^*) + I_\delta^2(u^*).
		\end{equation}
		
		First, we estimate the terminal perturbation term $I_\delta^1(u^*)$.
		By the Lipschitz continuity of the value function with respect to the spatial variable, there exists a constant $L_x>0$ such that
		\[
		|W(t+\delta, \widetilde{X}_{t+\delta}^{t,x;u^*}) - W(t+\delta,x)|
		\leq L_x \, |\widetilde{X}_{t+\delta}^{t,x;u^*} - x|.
		\]
		By the monotonicity and Lipschitz continuity of the semigroup with respect to the terminal condition, we obtain
		\[
		|I_\delta^1(u^*)| \leq C \cdot
		\mathbb{E}\!\left[ |\widetilde{X}_{t+\delta}^{t,x;u^*} - x| \,\big|\, \mathcal{F}_t \right].
		\]
		Applying the Cauchy--Schwarz inequality to the right-hand side yields
		\[
		|I_\delta^1(u^*)| \leq C \cdot
		\Big( \mathbb{E}\big[ |\widetilde{X}_{t+\delta}^{t,x;u^*} - x|^2 \,\big|\, \mathcal{F}_t \big] \Big)^{\frac{1}{2}}.
		\]
		
		By Lemma~\ref{lem:3.2}, there exists a constant $C>0$ such that
		\[
		\mathbb{E}\!\left[ |\widetilde{X}_{t+\delta}^{t,x;u^*} - x|^2 \,\big|\, \mathcal{F}_t \right]
		\leq C(1 + |x|^2)\delta.
		\]
		Hence,
		\begin{equation}\label{eq:3.24}
			|I_\delta^1(u^*)| \;\leq\; C(1 + |x|) \delta^{\frac{1}{2}}.
		\end{equation}
		
		Next, we estimate the generator integral term $I_\delta^2(u^*)$.
		Let $(Y_s', Z_s')$ be the solution of the backward stochastic differential equation (BSDE) with terminal condition $W(t+\delta,x)$.
		Then, by the representation of the backward semigroup, we have
		\[
		G_{t,t+\delta}^{t,x;u^*} \big[ W(t+\delta,x) \big]
		= Y_t'
		= W(t+\delta,x)
		+ \mathbb{E} \!\left[ \int_t^{t+\delta}
		f\big(s, \widetilde{X}_{s}^{t,x;u^*}, Y_s', Z_s', u_s^*\big)\, ds \,\Big|\, \mathcal{F}_t \right].
		\]
		Therefore, the generator integral term can be written as
		\[
		I_\delta^2(u^*)
		= \mathbb{E} \!\left[ \int_t^{t+\delta}
		f\big(s, \widetilde{X}_{s}^{t,x;u^*}, Y_s', Z_s', u_s^*\big)\, ds \,\Big|\, \mathcal{F}_t \right].
		\]
		Applying the Cauchy--Schwarz inequality yields
		\[
		|I_\delta^2(u^*)|
		\leq \delta^{\frac{1}{2}} \cdot
		\Bigg( \mathbb{E}\Big[ \int_t^{t+\delta}
		\big| f(s, \widetilde{X}_{s}^{t,x;u^*}, Y_s', Z_s', u_s^*) \big|^2 \, ds \,\Big|\, \mathcal{F}_t \Big] \Bigg)^{\frac{1}{2}}.
		\]
		
		Using the Lipschitz condition of $f$, we have
		\[
		|f(s,x,y,z,u)|^2 \leq C \Big( |f(s,x,0,0,u)|^2 + |y|^2 + \|z\|_{\ell^2}^2 \Big).
		\]
		By the linear growth condition in Assumption~\ref{ass:3.2}, there exists a constant $L>0$ such that
		\[
		|f(s,x,0,0,u)|^2 \leq 2L^2(1 + |x|^2).
		\]
		Combining with Lemma~\ref{lem:3.2}, there exists a constant $C_1>0$ such that
		\[
		\mathbb{E}\!\left[ \big| f(s, \widetilde{X}_{s}^{t,x;u^*}, 0, 0, u_s^*) \big|^2
		\,\Big|\, \mathcal{F}_t \right] \leq C_1(1 + |x|^2).
		\]
		Moreover, the moment estimates for the BSDE solution yield
		\[
		\mathbb{E}\!\left[ \int_t^{t+\delta} |Y_s'|^2 ds \,\Big|\, \mathcal{F}_t \right]
		+ \mathbb{E}\!\left[ \int_t^{t+\delta} \|Z_s'\|_{\ell^2}^2 ds \,\Big|\, \mathcal{F}_t \right]
		\leq C(1 + |x|^2)\delta.
		\]
		
		Combining the above estimates, we obtain
		\[
		\mathbb{E}\!\left[ \int_t^{t+\delta}
		\big| f(s, \widetilde{X}_{s}^{t,x;u^*}, Y_s', Z_s', u_s^*) \big|^2 ds
		\,\Big|\, \mathcal{F}_t \right]
		\leq C(1 + |x|^2)\delta.
		\]
		Hence,
		\begin{equation}\label{eq:3.25}
			|I_\delta^2(u^*)| \;\leq\; C(1 + |x|)\delta.
		\end{equation}
		
		By \eqref{eq:3.23}, \eqref{eq:3.24}, and \eqref{eq:3.25}, and noting that $\delta \leq \delta^{1/2}$ for $\delta \leq 1$, we obtain
		\[
		|W(t,x) - W(t+\delta,x)| \;\leq\; |I_\delta^1(u^*)| + |I_\delta^2(u^*)|
		\;\leq\; C(1 + |x|)\delta^{\frac{1}{2}}.
		\]
		Extending this result to any $t,t_0 \in [0,T]$ by taking $\delta = |t - t_0|$, the proof is complete.
	\end{proof}
	
	\section{The Existence of Viscosity Solutions}
	
	Let $W(t,x)$ denote the value function defined in \eqref{eq:3.2}. The associated HJB equation is
\begin{equation}\label{eq:4.1}
	\begin{cases}
		\partial_t W(t,x) + \displaystyle \inf_{u\in U} H(t,x,W,D_x W,D^2_x W,u) = 0, & (t,x)\in[0,T)\times \mathbb{R}^n,\\
		W(T,x) = \phi(x), & x\in \mathbb{R}^n,
	\end{cases}
\end{equation}
where $H$ is the Hamiltonian of the controlled L\'{e}vy jump system.

For any smooth function $v:[0,T]\times \mathbb{R}^n \to \mathbb{R}$, define
\begin{align}\label{eq:4.2}
	H(t,x,v,p,A,u) &= m_1 F(t,x,u)^\top p
	+ \frac{1}{2} \sigma^2 F(t,x,u)^\top A F(t,x,u) \nonumber\\
	&\quad + \int_{\mathbb{R}} \Bigl[v(t,x+F(t,x,u)\zeta) - v(t,x) - p^\top F(t,x,u)\zeta \Bigr] \nu(d\zeta) \nonumber\\
	&\quad + f\bigl(t,x,v(t,x),(\mathcal{L}^{u,(k)}v(t,x))_{k\ge1},u\bigr),
\end{align}
with the operator
\begin{equation}\label{eq:4.3}
	\mathcal{L}^{u,(k)}v(t,x)
	= \delta_{k1}\frac{1}{a_{11}} F(t,x,u)^\top D_x v(t,x)
	+ \int_{\mathbb{R}} \Bigl[v(t,x+F(t,x,u)\zeta) - v(t,x) - D_x v(t,x)^\top F(t,x,u)\zeta\Bigr] p_k(\zeta) \nu(d\zeta),
	\quad k\ge1.
\end{equation}

To simplify the presentation of viscosity solutions, we further define the \emph{Levy-type generator} for smooth functions $\varphi \in C^{1,2}([0,T]\times \mathbb{R}^n)$ as:
\begin{equation*}\label{eq:4.4}
	\begin{aligned}
		\mathcal{L}^u \varphi(t,x)
		&=
		m_1 F(t,x,u)^\top D_x\varphi(t,x)
		+ \tfrac{1}{2}\sigma^2 F(t,x,u)^\top D_x^2\varphi(t,x) F(t,x,u) \\
		&\quad
		+\int_{\mathbb R}
		\Bigl[
		\varphi\!\left(t,x+F(t,x,u)\zeta\right)
		-\varphi(t,x)
		- D_x\varphi(t,x)^\top F(t,x,u)\zeta
		\Bigr]\,\nu(d\zeta).
	\end{aligned}
\end{equation*}
$\text{Note that by substituting } v=\varphi, \, p=D_x\varphi, \, A=D_x^2\varphi \text{ into \eqref{eq:4.2}, the Hamiltonian can be rewritten as } H(t,x,\varphi,D_x\varphi,D_x^2\varphi,u) = \mathcal{L}^u \varphi(t,x) + f\bigl(t,x,\varphi(t,x),(\mathcal{L}^{u,(k)}\varphi(t,x))_{k\ge1},u\bigr).$

Next, we give the definition of viscosity solutions (see \cite{crandall1992users}).

\begin{defn}\label{defn:4.1}
	(i) A real-valued continuous function $W(\cdot,\cdot)\in C([0,T]\times \mathbb{R}^n)$ with at most polynomial growth is called a \emph{viscosity subsolution} (resp. \emph{viscosity supersolution}) of equation \eqref{eq:4.1} if
	\[
	W(T,x) \le \phi(x) \quad (\text{resp. } W(T,x) \ge \phi(x)), \quad \forall x\in \mathbb{R}^n,
	\]
	and for any function $\varphi \in C^{1,2}([0,T]\times \mathbb{R}^n)$ with at most quadratic growth, if at some point $(t,x)\in [0,T)\times \mathbb{R}^n$,
	\[
	W(t,x) = \varphi(t,x)
	\quad \text{and} \quad
	W-\varphi \text{ attains a local maximum (minimum) at } (t,x),
	\]
	then
	\[
	\partial_t \varphi(t,x) +\inf_{u\in U}
	\Bigl\{
	\mathcal L^u \varphi(t,x)
	+ f\Bigl(t,x,\varphi(t,x),
	(\mathcal{L}^{u,(k)}\varphi(t,x))_{k\ge1},u\Bigr)
	\Bigr\}\ge 0
	\quad (\text{resp. } \le 0).
	\]
	
	(ii) A real-valued continuous function $W(\cdot,\cdot)\in C([0,T]\times \mathbb{R}^n)$ is called a \emph{viscosity solution} of equation \eqref{eq:4.1} if it is both a viscosity subsolution and a viscosity supersolution.
\end{defn}
	
	For convenience in the following derivation, we denote the local increment
	\[
	\varphi^1(s,x,\zeta)
	:= \varphi(s, x + F(s,x,u)\,\zeta)
	- \varphi(s,x)
	- F(s,x,u)^\top D_x \varphi(s,x)\, \zeta .
	\]

	\begin{thm}\label{thm:4.1}
		Provided Assumption \ref{ass:3.1} holds, the value function $W(t,x)$ is a viscosity solution of the HJB equation \eqref{eq:4.1}.
	\end{thm}
	\begin{proof}
		Obviously, $W(T,x)=\phi(x)$ for all $x\in\mathbb{R}^n$. By Lemmas \ref{lem:3.3} and \ref{lem:3.5}, we know that
		$W(\cdot,\cdot)\in C([0,T]\times \mathbb{R}^n)$. We first prove that $W$ is a viscosity subsolution.
		For each given $(t,x_0)\in [0,T)\times \mathbb{R}^n$, suppose that
		$\varphi(\cdot,\cdot)\in C^{1,2}([0,T]\times \mathbb{R}^n)$
		such that
		\[
		\varphi(t,x_0) = W(t,x_0), \qquad \varphi \geq W \ \text{on} \ [0,T]\times \mathbb{R}^n .
		\]
		
		Consider the following FBSDE: for any $s\in [t,t+\delta]\subset [0,T]$,
		\begin{equation}\label{eq:4.4}
			\left\{
			\begin{aligned}
				d X_s^u &= F\left(s, X_{s-}^u, u_s\right)\, dL_s, \\[4pt]
				d Y_s^u &= -f\left(s, X_{s-}^u, Y_s^u, Z_s^u, u_s\right)\, ds
				+ \sum_{k=1}^{\infty} Z_s^{u,(k)}\, d H_s^{(k)}, \\[4pt]
				X_t^u &= x_0, \quad
				Y_{t+\delta}^u = \varphi\left(t+\delta, X_{t+\delta}^u\right),
			\end{aligned}
			\right.
		\end{equation}
		and
		\begin{equation}\label{eq:4.5}
			dY_s^{1,u} = -F_1\big(s,X_{s-}^{u},Y_s^{1,u},Z_s^{1,u},u_s\big)\, ds
			+ \sum_{k=1}^\infty Z_s^{1,u;(k)}\, dH_s^{(k)},
			\qquad
			Y_{t+\delta}^{1,u} = 0,
		\end{equation}
		where
		\begin{align*}
			F_1(s,x,y,z,u)
			&:= \partial_t \varphi(s,x)
			+ m_1 F(s,x,u)^\top D_x\varphi(s,x)
			+ \frac{1}{2}\, \sigma^{2}\, F(s,x,u)^\top D^2_x\varphi(s,x) \,F(s,x,u) \\
			&\quad
			+ \int_{\mathbb{R}} \varphi^1(s,x,\zeta) \, \nu(d\zeta)
			+ f\big(s,x, y+\varphi(s,x), z+h(s,x,u), u\big).
		\end{align*}
		
		Corresponding to the coefficient sequence of the Teugels martingale family,
		\[
		h(s,x,u)=\{\, h^{(k)}(s,x,u)\,\}_{k\ge 1},
		\]
		the $k$-th component can be written as
		\[
		h^{(k)}(s,x,u)
		= \delta_{k1}\, \frac{1}{a_{11}}\, F(s,x,u)^\top D_x\varphi(s,x)
		+ \int_{\mathbb{R}} \varphi^1(s,x,\zeta)\, p_k(\zeta)\, \nu(d\zeta).
		\]
	\end{proof}
	
	\begin{lem}\label{lem:4.2}
		For any $s \in [t,T]$, we have
		\[
		Y_s^{1,u} = Y_s^u - \varphi(s, X_s^u),
		\qquad
		Z_s^{1,u;(k)} = Z_s^{u,(k)} - h^{(k)}(s, X_{s-}^u, u_s),
		\]
		where
		\[
		h^{(k)}(s, X_{s-}^u, u_s)
		= \delta_{k1} \frac{1}{a_{11}} F(s, X_{s-}^{u}, u_s)^\top D_x\varphi(s, X_{s-}^{u})
		+ \int_{\mathbb{R}} \varphi^1(s, X_{s-}^u, \zeta) p_k(\zeta) \, \nu(d\zeta),
		\]
		$\varphi \in C^{1,2}([0,T]\times \mathbb{R}^n)$ is a smooth test function,
		$\{p_k\}_{k\ge 1}$ is the family of orthogonal polynomials with respect to the measure
		$\sigma^2 \delta_0(d\zeta) + \zeta^2 \nu(d\zeta)$, and
		$\{H^{(k)}\}_{k\ge 1}$ is the corresponding family of Teugels martingales.
	\end{lem}
	
	\begin{proof}
		Applying It\^o's formula for L\'{e}vy processes to $\varphi(s, X_s^u)$, we obtain
		\begin{align*}
			\varphi(T, X_T^u) - \varphi(t, X_t^u)
			&= \int_t^T \partial_s \varphi(s, X_{s-}^u) \, ds
			+ \int_t^T F(s, X_{s-}^{u}, u_s)^\top D_x\varphi(s, X_{s-}^{u}) \, dL_s \\
			&\quad + \frac{1}{2} \int_t^T \sigma^{2} F(s, X_{s-}^u, u_s)^\top D^2_x\varphi(s, X_{s-}^u)\,F(s, X_{s-}^u, u_s) \, ds \\
			&\quad + \sum_{t \le s \le T} \Bigl\{
			\varphi(s, X_s^u) - \varphi(s, X_{s-}^u)
			-  D_x \varphi(s, X_{s-}^u)^\top F(s, X_{s-}^u, u_s)\Delta L_s
			\Bigr\}.
		\end{align*}
		
		Define the local increment
		\[
		\varphi^1(s, X_{s-}^u, \zeta)
		:= \varphi(s, X_{s-}^u + F(s, X_{s-}^u, u_s)\zeta)
		- \varphi(s, X_{s-}^u)
		- D_x \varphi(s, X_{s-}^u)^\top F(s, X_{s-}^u, u_s) \zeta,
		\]
		then the jump sum can be written as
		\[
		\sum_{t \le s \le T} \varphi^1(s, X_{s-}^u, \Delta L_s).
		\]
		
		Using the Teugels martingale expansion, the jump sum decomposes into a martingale integral and a compensator:
		\[
		\sum_{t \le s \le T} \varphi^1(s, X_{s-}^u, \Delta L_s)
		= \sum_{k=1}^{\infty} \int_t^T
		\left( \int_{\mathbb{R}} \varphi^1(s, X_{s-}^u, \zeta) p_k(\zeta) \, \nu(d\zeta) \right) dH_s^{(k)}
		+ \int_t^T \int_{\mathbb{R}} \varphi^1(s, X_{s-}^u, \zeta) \, \nu(d\zeta) \, ds.
		\]
		
		Substituting the decomposition of the L\'{e}vy process
		\[
		dL_s = \frac{1}{a_{11}} \, dH_s^{(1)} + m_1 \, ds
		\]
		into the continuous integral term
		$\int_t^T F(s, X_{s-}^{u}, u_s)^\top D_x\varphi(s, X_{s-}^{u}) \, dL_s$, we get
		\[
		\int_t^T F(s, X_{s-}^{u}, u_s)^\top D_x\varphi(s, X_{s-}^{u}) \, dL_s
		= \int_t^T \frac{1}{a_{11}} F(s, X_{s-}^{u}, u_s)^\top D_x\varphi(s, X_{s-}^{u}) \, dH_s^{(1)}
		+ \int_t^T m_1 F(s, X_{s-}^{u}, u_s)^\top D_x\varphi(s, X_{s-}^{u}) \, ds.
		\]
		
		Collecting the continuous drift term, compensator term, and martingale terms, we have
		\begin{align*}
			d\varphi(s, X_{s-}^u)
			&= \Bigl( \partial_s \varphi(s, X_{s-}^u)
			+ m_1 F(s, X_{s-}^u, u_s)^\top D_x\varphi(s, X_{s-}^u) \\
			&\quad + \frac{1}{2} \sigma^{2} F(s, X_{s-}^u, u_s)^\top D^2_x\varphi(s, X_{s-}^u)\,F(s, X_{s-}^u, u_s)
			+ \int_{\mathbb{R}} \varphi^1(s, X_{s-}^u, \zeta) \, \nu(d\zeta) \Bigr) ds \\
			&\quad + \frac{1}{a_{11}} F(s, X_{s-}^u, u_s)^\top D_x\varphi(s, X_{s-}^{u}) \, dH_s^{(1)}
			+ \sum_{k=1}^{\infty} \left( \int_{\mathbb{R}} \varphi^1(s, X_{s-}^u, \zeta) p_k(\zeta) \, \nu(d\zeta) \right) dH_s^{(k)}.
		\end{align*}
		
		Now, substitute the above expansion into the original BSDE for $Y_s^u$:
		\[
		dY_s^{u} = -f(s, X_{s-}^{u}, Y_{s}^{u}, Z_s^{u}, u_s) \, ds
		+ \sum_{k=1}^{\infty} Z_s^{u,(k)} \, dH_s^{(k)},
		\]
		we obtain the deviation for $Y_s^u - \varphi(s, X_s^u)$:
		\[
		\begin{aligned}
			dY_s^{u} - d\varphi(s, X_{s-}^u)
			&= -\Bigl(
			f(s, X_{s-}^{u}, Y_{s}^{1,u} + \varphi(s, X_{s-}^{u}), Z_s^{1,u} + h(s, X_{s-}^{u}, u_s), u_s)
			+ \partial_t \varphi(s, X_{s-}^{u}) \\
			&\quad + m_1 F(s, X_{s-}^{u}, u_s)^\top D_x\varphi(s, X_{s-}^{u})
			+ \frac{1}{2} \sigma^{2} F(s, X_{s-}^u, u_s)^\top D^2_x\varphi(s,X_{s-}^u)\,F(s, X_{s-}^u, u_s) \\
			&\quad + \int_{\mathbb{R}} \varphi^1(s, X_{s-}^{u}, \zeta) \, \nu(d\zeta)
			\Bigr) \, ds \\
			&\quad + \sum_{k=1}^{\infty} \left[ Z_s^{u,(k)} - \left(
			\delta_{k1} \frac{1}{a_{11}} F(s, X_{s-}^{u}, u_s)^\top D_x\varphi(s, X_{s-}^{u})
			+ \int_{\mathbb{R}} \varphi^1(s, X_{s-}^u, \zeta) p_k(\zeta) \, \nu(d\zeta)
			\right) \right] dH_s^{(k)}.
		\end{aligned}
		\]
		
		This matches the BSDE for $Y_s^{1,u}$ with drift function
		\[
		\begin{aligned}
			F_1(s, X_{s-}^{u}, Y_{s}^{1,u}, Z_s^{1,u}, u_s)
			&:= \partial_t \varphi(s, X_{s-}^{u})
			+ m_1 F(s, X_{s-}^{u}, u_s)^\top D_x\varphi(s, X_{s-}^{u}) \\
			&\quad + \frac{1}{2} \sigma^{2} F(s, X_{s-}^u, u_s)^\top D^2_x\varphi(s, X_{s-}^u)\,F(s,X_{s-}^u, u_s)
			+ \int_{\mathbb{R}} \varphi^1(s, X_{s-}^{u}, \zeta) \, \nu(d\zeta) \\
			&\quad + f(s, X_{s-}^{u}, Y_{s}^{1,u} + \varphi(s, X_{s-}^{u}), Z_s^{1,u} + h(s, X_{s-}^{u}, u_s), u_s),
		\end{aligned}
		\]
		and martingale coefficients
		\[
		Z_s^{1,u;(k)}
		= Z_s^{u,(k)}
		- \left(
		\delta_{k1} \frac{1}{a_{11}} F(s, X_{s-}^{u}, u_s)^\top D_x\varphi(s, X_{s-}^{u})
		+ \int_{\mathbb{R}} \varphi^1(s, X_{s-}^u, \zeta) p_k(\zeta) \, \nu(d\zeta)
		\right).
		\]
		
		The coefficient sequence $h(s, X_{s-}^u, u_s) = \{ h^{(k)}(s, X_{s-}^u, u_s) \}_{k \ge 1}$ is defined by
		\[
		h^{(k)}(s, X_{s-}^u, u_s)
		= \delta_{k1} \frac{1}{a_{11}} F(s, X_{s-}^{u}, u_s)^\top D_x\varphi(s, X_{s-}^{u})
		+ \int_{\mathbb{R}} \varphi^1(s, X_{s-}^u, \zeta) p_k(\zeta) \, \nu(d\zeta).
		\]
		
		The terminal condition at $s = t + \delta$ is
		\[
		Y_{t+\delta}^u = \varphi(t+\delta, X_{t+\delta}^u) \implies Y_{t+\delta}^{1,u} = Y_{t+\delta}^u - \varphi(t+\delta, X_{t+\delta}^u) = 0,
		\]
		which matches BSDE \eqref{eq:4.4}. By uniqueness of solutions, we conclude that
		\[
		Y_s^{1,u} = Y_s^u - \varphi(s, X_s^u),
		\]
		and
		\[
		Z_s^{1,u;(k)} = Z_s^{u,(k)} - h^{(k)}(s, X_{s-}^u, u_s).
		\]
		This completes the proof.
	\end{proof}
	
	Consider the following BSDE: for all $s \in [t, t+\delta]$,
	\begin{equation}\label{eq:4.6}
		dY_s^{2,u} = -F_1(s, x_0, 0, 0, u_s)\, ds + \sum_{k=1}^\infty Z_s^{2,u;(k)}\, dH_s^{(k)},
		\quad Y_{t+\delta}^{2,u} = 0.
	\end{equation}
	We have the following estimate.
	
	\begin{lem}\label{lem:4.3}
		For each given $v \in \mathcal{U}^{t}[t, t + \delta]$, it holds that
		\[
		\lvert Y_t^{1,v} - Y_t^{2,v} \rvert \le C \delta^{\frac{3}{2}},
		\]
		where $C$ is a positive constant depending on $x_0$, but independent of $v$ and $\delta$.
	\end{lem}
	
	\begin{proof}
		Recall that $(Y_s^{1,v}, Z_s^{1,v})$ and $(Y_s^{2,v}, Z_s^{2,v})$ are solutions to BSDEs \eqref{eq:4.5} and \eqref{eq:4.6} respectively, with the common terminal condition $Y_{t+\delta}^{1,v} = Y_{t+\delta}^{2,v} = 0$.
		By the standard representation formula for solutions to BSDEs with zero terminal value, we can express the initial values as
		\[
		\begin{aligned}
			Y_t^{1,v} &= \mathbb{E}\Bigg[ \int_t^{t+\delta}
			F_1(s, X_s^v, Y_s^{1,v}, Z_s^{1,v}, v_s)\, ds \Bigg], \\[6pt]
			Y_t^{2,v} &= \mathbb{E}\Bigg[ \int_t^{t+\delta}
			F_1(s, x_0, 0, 0, v_s)\, ds \Bigg].
		\end{aligned}
		\]
		Taking the absolute difference of the two equalities and applying the triangle inequality for integrals and expectations, we arrive at
		\[
		|Y_t^{1,v} - Y_t^{2,v}| \le \mathbb{E}\Bigg[ \int_t^{t+\delta} \hat{F}_s \, ds \Bigg],
		\]
		where we define the pointwise difference
		\[
		\hat{F}_s = \big| F_1(s, X_s^v, Y_s^{1,v}, Z_s^{1,v}, v_s) - F_1(s, x_0, 0, 0, v_s) \big|.
		\]
		
		Since $\varphi \in C^{1,2}([0,T] \times \mathbb{R}^n)$ by assumption, the driver function $F_1$ is uniformly Lipschitz continuous with respect to its state variables $(x, y, z)$.
		Therefore, there exists a constant $C > 0$ such that
		\[
		\hat{F}_s \le C \Big( |X_s^v - x_0| + |Y_s^{1,v}| + \|Z_s^{1,v}\|_{\ell^2} \Big)
		\]
		holds for all $s \in [t, t+\delta]$.
		
		We now introduce the perturbed processes to simplify the estimation:
		\[
		\tilde{X}_s^v = X_s^v - x_0,\quad \tilde{Y}_s^v = Y_s^{1,v},\quad \tilde{Z}_s^v = Z_s^{1,v}.
		\]
		By the definition of the state process $X_s^v$, the triple $(\tilde{X}_s^v, \tilde{Y}_s^v, \tilde{Z}_s^v)$ satisfies the following forward-backward stochastic differential system:
		\[
		\left\{
		\begin{aligned}
			d\tilde{X}_s^v &= F\left(s,\, \tilde{X}_{s-}^v + x_0,\, v_s\right)\, dL_s, \\[6pt]
			d\tilde{Y}_s^v &= -F_1\left(s,\, \tilde{X}_{s-}^v + x_0,\, \tilde{Y}_s^v,\, \tilde{Z}_s^v,\, v_s\right)\, ds
			+ \sum_{k=1}^{\infty} \tilde{Z}_s^{v,(k)}\, dH_s^{(k)}, \\[6pt]
			\tilde{X}_t^v &= 0,\qquad \tilde{Y}_{t+\delta}^v = 0.
		\end{aligned}
		\right.
		\]
		
		With the help of the a priori estimates established in Lemmas \ref{lem:2.4} and \ref{lem:2.5}, we obtain the following key moment estimate for the perturbed system:
		\[
		\begin{aligned}
			&\mathbb{E}\left[
			\sup_{t\le s\le t+\delta}\bigl(|\tilde{X}_s^v|^2 + |\tilde{Y}_s^v|^2\bigr)
			+ \int_t^{t+\delta} \|\tilde{Z}_s^v\|_{\ell^2}^2\ ds
			\right] \\
			&\qquad \le C \left\{
			\mathbb{E}\int_t^{t+\delta} |F(s, x_0, v_s)|^2 ds
			+ \mathbb{E}\left[
			\int_t^{t+\delta} |F_1(s, x_0, 0, 0, v_s)|^2 ds
			\right]
			\right\} \\
			&\qquad \le C \delta.
		\end{aligned}
		\]
		
		We first estimate the integral of $|\tilde{X}_s^v|$ over $[t, t+\delta]$.
		Applying the Cauchy--Schwarz inequality for integrals, we have
		\[
		\mathbb{E}\!\left[\int_t^{t+\delta} |\tilde X_s^v|\,ds\right]
		\le
		\delta^{1/2}\,
		\mathbb{E}\!\left[\left(\int_t^{t+\delta} |\tilde X_s^v|^2\,ds\right)^{1/2}\right].
		\]
		Using the bound $\int_t^{t+\delta} |\tilde X_s^v|^2 ds \le \delta \sup_{t\le r\le t+\delta} |\tilde X_r^v|^2$ and taking the expectation, we further get
		\[
		\mathbb{E}\!\left[\int_t^{t+\delta} |\tilde X_s^v|\,ds\right]
		\le
		\delta\,\mathbb{E}\!\left[\sup_{t\le r\le t+\delta} |\tilde X_r^v|\right]
		\le
		\delta \left(\mathbb{E}\!\left[\sup_{t\le r\le t+\delta} |\tilde X_r^v|^2\right]\right)^{1/2}.
		\]
		Combining this with the moment estimate above, we conclude
		\[
		\mathbb{E}\!\left[\int_t^{t+\delta} |\tilde X_s^v|\,ds\right]
		\le C\,\delta^{3/2}.
		\]
		Following exactly the same reasoning, we obtain the estimate for the $\tilde{Y}$-component:
		\[
		\mathbb{E}\!\left[\int_t^{t+\delta} |\tilde Y_s^v|\,ds\right]
		\le C\,\delta^{3/2}.
		\]
		
		Next, we focus on estimating the $\tilde{Z}$-component.
		Applying It\^o's formula to $|Y_s^{1,v}|^2$ and taking expectation, note that the martingale part vanishes and the terminal term satisfies $|Y_{t+\delta}^{1,v}|^2 = 0$, so that
		\[
		\begin{aligned}
			\mathbb{E}\left[ \int_{t}^{t+\delta}\|Z_s^{1,v}\|_{\ell^2}^2\, ds \right]
			&= -|Y_t^{1,v}|^2
			+ \mathbb{E}\left[ \left( \int_{t}^{t+\delta}
			F_1\bigl(s, X_s^v, Y_s^{1,v}, Z_s^{1,v}, v_s\bigr)\, ds \right)^2 \right] \\
			&\le 2\,\mathbb{E}\left[
			\left( \int_{t}^{t+\delta}
			\bigl|F_1\bigl(s, X_s^v, Y_s^{1,v}, Z_s^{1,v}, v_s\bigr)\bigr|
			\, ds \right)^2 \right],
		\end{aligned}
		\]
		where we have used the elementary inequality $|a+b|^2 \le 2|a|^2 + 2|b|^2$ to bound the right-hand side.
		
		The driver function $F_1$ satisfies the linear growth condition, i.e., there exists a constant $C>0$ such that
		\[
		|F_1(s, X_s^v, Y_s^{1,v}, Z_s^{1,v}, v_s)|
		\le C\bigl(1 + |X_s^v| + |Y_s^{1,v}| + \|Z_s^{1,v}\|_{\ell^2}\bigr).
		\]
		Squaring both sides and using the inequality $(a+b+c+d)^2 \le 4(a^2+b^2+c^2+d^2)$, we obtain
		\[
		|F_1(s, X_s^v, Y_s^{1,v}, Z_s^{1,v}, v_s)|^2
		\le C\bigl(1 + |X_s^v|^2 + |Y_s^{1,v}|^2 + \|Z_s^{1,v}\|_{\ell^2}^2\bigr).
		\]
		
		Applying the Cauchy--Schwarz inequality once more to the integral term, we have
		\[
		\begin{aligned}
			&\mathbb{E}\left[\left(\int_t^{t+\delta}
			|F_1(s, X_s^v, Y_s^{1,v}, Z_s^{1,v}, v_s)|\,ds\right)^2\right] \\
			&\le
			\delta \,\mathbb{E}\left[
			\int_t^{t+\delta}
			|F_1(s, X_s^v, Y_s^{1,v}, Z_s^{1,v}, v_s)|^2\,ds
			\right] \\
			&\le C\delta \left(
			\delta + \mathbb{E}\int_t^{t+\delta}
			\bigl(|X_s^v|^2 + |Y_s^{1,v}|^2 + \|Z_s^{1,v}\|_{\ell^2}^2\bigr)\,ds
			\right) \\
			&\le C\delta^2.
		\end{aligned}
		\]
		
		Substituting this back into the earlier estimate for the $Z$-integral yields
		\[
		\mathbb{E}\Big[\int_t^{t+\delta} \|Z_s^{1,v}\|_{\ell^2}^2 ds\Big]
		\le C \delta^2.
		\]
		Using the Cauchy--Schwarz inequality one last time, we get the $L^1$-estimate for $Z_s^{1,v}$:
		\[
		\mathbb{E}\Big[\int_t^{t+\delta} \|Z_s^{1,v}\|_{\ell^2} ds\Big]
		\le \left(\mathbb{E}\Big[\int_t^{t+\delta} \|Z_s^{1,v}\|_{\ell^2}^2 ds\Big]\right)^{\frac{1}{2}} \cdot \delta^{\frac{1}{2}}
		\le C \delta^{\frac{3}{2}}.
		\]
		
		Finally, we combine all three estimates for $\tilde{X}$, $\tilde{Y}$ and $\tilde{Z}$.
		Recall that
		\[
		|Y_t^{1,v} - Y_t^{2,v}| \le C \mathbb{E}\left[ \int_t^{t+\delta} \Big( |\tilde{X}_s^v| + |\tilde{Y}_s^v| + \|\tilde{Z}_s^v\|_{\ell^2} \Big) ds \right].
		\]
		Substituting the respective integral estimates into this inequality, we obtain
		\[
		|Y_t^{1,v} - Y_t^{2,v}| \le C \delta^{\frac{3}{2}},
		\]
		which is exactly the estimate we set out to prove.
	\end{proof}
	
	Now we compute $\inf_{v \in \mathcal{U}^{t}[t, t+\delta]} Y^{2,v}_{t}$.
	
	\begin{lem}\label{lem:4.4}
		We have
		\[
		Y_t^0 = \inf_{v \in \mathcal{U}^{t}[t,t+\delta]} Y_t^{2,v},
		\]
		where $Y_t^0$ is the solution of the following ordinary differential equation:
		\[
		dY_s^0 = -F_0(s,x_0)\,ds, \qquad Y_{t+\delta}^0 = 0, \qquad s \in [t,t+\delta],
		\]
		and
		\[
		F_0(s,x_0) := \inf_{u\in U} F_1(s,x_0,0,0,u).
		\]
	\end{lem}
	
	\begin{proof}
		For any given $v \in \mathcal{U}^{t}[t,t+\delta]$, by the definition of the infimum, we have
		\[
		F_1(s,x_0,0,0,v_s) \ge F_0(s,x_0), \quad \forall s \in [t, t+\delta].
		\]
		This implies $-F_1(s,x_0,0,0,v_s) \le -F_0(s,x_0)$. Since both the BSDE for $Y^{2,v}$ and the ODE for $Y^0$ share the same terminal condition $Y_{t+\delta} = 0$, we can apply the comparison theorem for BSDEs (Lemma \ref{lem:2.3}) to conclude
		\[
		Y_t^{2,v} \ge Y_t^0.
		\]
		
		On the other hand, by the measurable selection theorem, there exists a deterministic control $\mu \in \mathcal{U}^{t}[t,t+\delta]$ such that
		\[
		F_0(s,x_0) = F_1(s,x_0,0,0,\mu_s), \quad \forall s \in [t, t+\delta].
		\]
		For this control $\mu$, the BSDE for $Y^{2,\mu}$ coincides with the ODE for $Y^0$, and thus $Y_t^{2,\mu} = Y_t^0$. This shows that the infimum is attainable, and the conclusion follows.
	\end{proof}
	
	By the dynamic programming principle (Theorem \ref{thm:3.5}), we have
	\[
	W(t, x_0) = \inf_{v \in \mathcal{U}^{t}[t, t+\delta]} G_{t,t+\delta}^{t,x_0;v} \Bigl[ W\bigl(t+\delta, X_{t+\delta}^{t,x_0;v}\bigr) \Bigr].
	\]
	
	Since $\varphi(t+\delta,\cdot) \ge W(t+\delta,\cdot)$ in a neighborhood of $x_0$, and by the monotonicity of the BSDE solution operator (Theorem 3.1 in \cite{elotmani2008backward}), it follows that for each $v \in \mathcal{U}^{t}[t,t+\delta]$,
	\[
	Y_t^v = G_{t,t+\delta}^{t,x_0;v} \Bigl[ \varphi\bigl(t+\delta, X_{t+\delta}^{t,x_0;v}\bigr) \Bigr] \ge G_{t,t+\delta}^{t,x_0;v} \Bigl[ W\bigl(t+\delta, X_{t+\delta}^{t,x_0;v}\bigr) \Bigr].
	\]
	Taking the infimum over all $v \in \mathcal{U}^{t}[t,t+\delta]$, we obtain
	\[
	\inf_{v \in \mathcal{U}^{t}[t,t+\delta]} Y_t^v \ge W(t, x_0).
	\]
	Recall from Lemma \ref{lem:4.2} that $Y_t^{1,v} = Y_t^v - \varphi(t, x_0)$. Since $W - \varphi$ attains a local maximum at $(t, x_0)$, we have $\varphi(t, x_0) = W(t, x_0)$. Therefore,
	\[
	\inf_{v \in \mathcal{U}^{t}[t,t+\delta]} Y_t^{1,v} = \inf_{v \in \mathcal{U}^{t}[t,t+\delta]} \bigl[Y_t^v - \varphi(t,x_0)\bigr] \ge W(t, x_0) - \varphi(t, x_0) = 0.
	\]
	
	By the uniform error estimate in Lemma \ref{lem:4.3}, we have
	\[
	Y_t^{1,v} \ge Y_t^{2,v} - C\delta^{3/2}, \quad \forall v \in \mathcal{U}^{t}[t,t+\delta].
	\]
	Taking the infimum over $v$ gives
	\[
	\inf_{v \in \mathcal{U}^{t}[t,t+\delta]} Y_t^{2,v} \ge \inf_{v \in \mathcal{U}^{t}[t,t+\delta]} Y_t^{1,v} - C\delta^{3/2} \ge -C\delta^{3/2}.
	\]
	Combining this with Lemma \ref{lem:4.4} yields
	\[
	Y_t^0 \ge -C \delta^{3/2}.
	\]
	Integrating the ODE for $Y_s^0$, we find
	\[
	Y_t^0 = \int_t^{t+\delta} F_0(s,x_0)\, ds.
	\]
	Substituting this in, we get
	\[
	\frac{1}{\delta} \int_t^{t+\delta} F_0(s,x_0)\, ds = \frac{1}{\delta} Y_t^0 \ge -C\delta^{1/2}.
	\]
	Letting $\delta \to 0$, the left-hand side converges to $F_0(t, x_0)$ by the fundamental theorem of calculus, and we obtain
	\[
	F_0(t, x_0) \ge 0.
	\]
	Recalling that $F_0(t,x_0) = \inf_{u \in U} F_1(t, x_0, 0, 0, u)$, this implies
	\[
	\inf_{u \in U} F_1(t, x_0, 0, 0, u) \ge 0.
	\]
	By the definition of the Hamiltonian operator,
	\[
	F_1(t, x_0, 0, 0, u) = \partial_t \varphi(t, x_0)
	+ \mathcal L^u \varphi(t,x_0)
	+ f\Bigl(t,x_0,\varphi(t,x_0),
	\bigl(\mathcal{L}^{u,(k)}\varphi(t,x_0)\bigr)_{k\ge1},u\Bigr).
	\]
	Therefore,
	\[
	\partial_t \varphi(t, x_0) +\displaystyle\inf_{u\in U}
	\Bigl\{
	\mathcal L^u \varphi(t,x_0)
	+ f\Bigl(t,x_0,\varphi(t,x_0),
	\bigl(\mathcal{L}^{u,(k)}\varphi(t,x_0)\bigr)_{k\ge1},u\Bigr)
	\Bigr\} \ge 0,
	\]
	which means that $W$ is a viscosity subsolution.
	
	We now prove the viscosity supersolution property.
	
	Let $\varphi \in C^{1,2}$ and assume that $W-\varphi$ attains a local minimum at $(t,x)$ with $W(t,x)=\varphi(t,x)$. Then, for $s\in[t,t+\delta]$,
	\[
	W(s,\widetilde X_s^{t,x;u}) \ge \varphi(s,\widetilde X_s^{t,x;u}),
	\]
	and in particular at $s=t+\delta$. By the comparison theorem for BSDEs and the dynamic programming principle, we have
	\[
	W(t,x)\le G_{t,t+\delta}^{t,x;u}\big[\varphi(t+\delta,\widetilde X_{t+\delta}^{t,x;u})\big].
	\]
	
	Repeating the argument in the subsolution case, with the inequalities reversed, and taking the limit as $\delta\to0$, we obtain the required HJB inequality for the supersolution.

	\section{The Uniqueness of Viscosity Solutions}
	Before proving the uniqueness of viscosity solutions, we briefly clarify the definition of viscosity solutions adopted in Section~4.
	
	It should be emphasized that, in Definition~4.1, the nonlocal term is formulated directly as a full-space integral with respect to theL\'{e}vy measure, without decomposing the jump term into small and large jumps.
	
	In the existing literature, the nonlocal operator is often treated by decomposing theL\'{e}vy jump term into small and large jumps, and the corresponding parts are applied separately to the test function and the value function in the definition of viscosity solutions (see, e.g., \cite{buckdahn2010integral}). Under suitable regularity and integrability conditions on the L\'{e}vy measure and the test functions, such a split formulation can be shown to be equivalent to a unified full-space integral representation (see, e.g., \cite{barles2008second,buckdahn2010integral}).
	
	In contrast to this approach, in the present paper we adopt from the outset a unified formulation based on the full-spaceL\'{e}vy integral operator. It is therefore necessary to verify that this nonlocal operator is well defined under the model assumptions and can be directly employed in the subsequent analysis, in particular in the derivation of the comparison principle.
	
	This is indeed the case in our framework. On the one hand, the L\'{e}vy measure satisfies a second-moment condition; on the other hand, the value function is globally Lipschitz continuous with respect to the spatial variable (see Lemma~\ref{lem:3.3}). These properties jointly ensure the integrability and stability of the nonlocal operator under the full-space formulation. Consequently, Definition 4.1 is mathematically well posed within the viscosity solution framework and is compatible with the general theory of viscosity solutions for integro-differential equations developed in  \cite{barles2008second}.
	
	Based on the above analysis, this section is devoted to the study of the uniqueness of viscosity solutions to the HJB equation \eqref{eq:4.1}.

	For the convenience of the subsequent analysis, we next present a key technical lemma. This lemma characterizes the increment control of the value function in the jump direction and provides a basis for the integrability analysis of the nonlocal term in the sense of full-space integration.
	
	\begin{lem}\label{lem:5.1}
		Under Assumption~\ref{ass:3.1}, and suppose that the value function $W(t,x)$ satisfies the properties stated in Lemma~\ref{lem:3.3}, namely, $W$ is globally Lipschitz continuous with respect to the spatial variable $x$ and has linear growth. Then there exists a constant $C>0$ such that for any $(t,x)\in[0,T]\times\mathbb{R}^n$ and any $\zeta\in\mathbb{R}$, we have
		\[
		\bigl| W(t,x+F(t,x,u)\zeta)-W(t,x) \bigr|
		\le C\,|F(t,x,u)\zeta|.
		\tag{5.1}\label{eq:5.1}
		\]
	\end{lem}
	
	\begin{proof}
		By Lemma~\ref{lem:3.3}, there exists a constant $C>0$ such that for any $x_1,x_2\in\mathbb{R}^n$,
		\[
		\bigl|W(t,x_1)-W(t,x_2)\bigr|
		\le C\,|x_1-x_2|.
		\]
		Taking $x_1=x+F(t,x,u)\zeta$ and $x_2=x$, we obtain
		\[
		\bigl| W(t,x+F(t,x,u)\zeta)-W(t,x) \bigr|
		\le C\,|F(t,x,u)\zeta|,
		\]
		which completes the proof.
	\end{proof}
	
	\begin{rmk}
		Lemma~\ref{lem:5.1} characterizes the increment control of the value function in the jump direction and plays a fundamental role in the nonlocal operator in the full-space integral form adopted in this paper. Under the assumption that the L\'{e}vy measure satisfies a second-moment condition, this lemma guarantees a linear bound of
		\(
		|W(t,x+F(t,x,u)\zeta)-W(t,x)|
		\)
		with respect to the jump size $\zeta$ in the large-jump region, thereby ensuring the integrability of the corresponding integral term.
		
		On the other hand, the integral term over the small-jump region can be controlled by the second-order smoothness of the test function
		$\varphi\in C^{1,2}$
		together with the second-moment condition of the L\'{e}vy measure. Therefore, the nonlocal operator in the full-space integral sense is well defined, and consequently, the definition of viscosity solutions adopted in this paper is mathematically justified.
	\end{rmk}
	
	\begin{lem}\label{lem:5.2}
		Let $W_1$ be a viscosity subsolution of the HJB equation \eqref{eq:4.1}, and let $W_2$ be the corresponding viscosity supersolution. Suppose that
		$W_1, W_2 \in C\bigl([0,T]\times\mathbb{R}^n\bigr)$
		and both have linear growth. Furthermore, assume that Assumption~\ref{ass:3.1} holds and that the L\'{e}vy measure $\nu$ satisfies the second-moment condition
		\[
		\int_{\mathbb{R}} (1 \wedge |x|^2)\, \nu(dx) < +\infty.
		\]
		Define the difference function
		\[
		\omega := W_1 - W_2 .
		\]
		Then $\omega$ is a viscosity subsolution of the following auxiliary equation:
		\begin{equation*}\tag{5.2}\label{eq:5.2}
			\begin{cases}
				\displaystyle
				\partial_t \omega(t,x)
				+
				\sup_{u\in U}
				\Bigl\{
				\mathcal{L}^u \omega(t,x)
				+
				L_2 \lvert \omega(t,x) \rvert
				+
				L_3 \bigl\| \bigl( \mathcal{L}^{u,(k)} \omega(t,x) \bigr)_{k\ge 1} \bigr\|_{\ell^2}
				\Bigr\}
				= 0,
				& (t,x)\in [0,T)\times\mathbb{R}^n,\\[0.8em]
				\omega(T,x)=0,
				& x\in\mathbb{R}^n.
			\end{cases}
		\end{equation*}
		Here the constants $L_2, L_3 > 0$ depend only on the Lipschitz constants of the function $f$ in Assumption~\ref{ass:3.1}.
	\end{lem}
	
	\begin{proof}
		Let $\phi \in C^{1,2}([0,T]\times \mathbb{R}^n)$ with at most quadratic growth, and assume that $(t_0,x_0)$ is a local maximum point of $\omega - \phi$.
		
		Take $\varepsilon,\alpha,\eta>0$ and define the function of four variables
		\[
		\Phi_{\varepsilon,\alpha}(t,x,s,y)
		:= W_1(t,x) - W_2(s,y) - \phi(t,x) - \psi(t,x,s,y),
		\tag{5.3}\label{eq:5.3}
		\]
		where
		\[
		\psi(t,x,s,y) := \frac{|x-y|^2}{\varepsilon^2} + \frac{|t-s|^2}{\alpha^2} + \eta(|x|^2+|y|^2).
		\tag{5.4}\label{eq:5.4}
		\]
		Since $W_1,W_2$ have at most polynomial growth and the penalty term $-\eta(|x|^2+|y|^2)\to -\infty$,
		the function $\Phi_{\varepsilon,\alpha}$ attains a maximum on
		$[0,T]\times\mathbb{R}^n\times[0,T]\times\mathbb{R}^n$.
		Denote a maximum point by $(\bar t,\bar x,\bar s,\bar y)$.
		
		By the definition of the maximum point, for any $(t,x)$,
		\[
		\begin{aligned}
			&W_1(\bar t,\bar x) - W_2(\bar s,\bar y) - \phi(\bar t,\bar x)
			- \frac{|\bar x-\bar y|^2}{\varepsilon^2} - \frac{|\bar t-\bar s|^2}{\alpha^2} - \eta(|\bar x|^2+|\bar y|^2) \\
			&\quad \ge W_1(t,x) - W_2(t,x) - \phi(t,x) - 2\eta |x|^2 .
		\end{aligned}
		\]
		
		Taking $x=0$ and defining
		\[
		C_0 := \sup_{t\in[0,T]} \bigl| W_1(t,0) - W_2(t,0) - \phi(t,0) \bigr|,
		\]
		by the continuity of $W_1,W_2$ and the boundedness of $\phi$, it is easy to see that $C_0<\infty$.
		
		Thus,
		\[
		\frac{|\bar x-\bar y|^2}{\varepsilon^2} + \frac{|\bar t-\bar s|^2}{\alpha^2}
		\le W_1(\bar t,\bar x) - W_2(\bar s,\bar y) - \phi(\bar t,\bar x) - \eta(|\bar x|^2+|\bar y|^2) - C_0 .
		\]
		
		Since $W_1$ and $W_2$ have at most polynomial growth, the test function
		$\phi \in C^{1,2}([0,T]\times\mathbb{R}^n)$ is fixed, and the penalty term
		$-\eta(|\bar x|^2+|\bar y|^2)$ is quadratic, the right-hand side is bounded.
		Otherwise, this would contradict the existence of a maximum point of
		$\Phi_{\varepsilon,\alpha}$.
		
		Therefore, there exists a constant $C>0$, independent of $\varepsilon,\alpha,\eta$, such that
		\begin{equation}\tag{5.5}\label{eq:5.5}
			\frac{|\bar x-\bar y|^2}{\varepsilon^2}
			+
			\frac{|\bar t-\bar s|^2}{\alpha^2}
			\le C .
		\end{equation}
		
		As $\varepsilon\to0$ and $\alpha\to0$, we have
		$(\bar t,\bar x,\bar s,\bar y)\to(t_0,x_0,t_0,x_0)$.
		
		By the Crandall--Ishii type lemma for the IPDE case
		(see Theorem~2.2 in \cite{jakobsen2005continuous}),
		there exist symmetric matrices $X,Y\in\mathbb{S}^n$ such that
		\[
		(\partial_t\psi,D_x\psi,X)
		\in\overline{\mathcal {P}}^{2,+}W_1(\bar t,\bar x),\qquad
		(-\partial_s\psi,-D_y\psi,Y)
		\in\overline{\mathcal {P}}^{2,-}W_2(\bar s,\bar y),
		\]
		and the matrix inequality
		\begin{equation}\tag{5.6}\label{eq:5.6}
			\begin{pmatrix}
				X & 0\\
				0 & -Y
			\end{pmatrix}
			\le
			\frac{2}{\varepsilon^2}
			\begin{pmatrix}
				I & -I\\
				-I & I
			\end{pmatrix}
			+
			2\eta
			\begin{pmatrix}
				I & 0\\
				0 & I
			\end{pmatrix}
		\end{equation}
		holds.
		
		On the other hand, since $(\bar t,\bar x,\bar s,\bar y)$ is a maximum point of
		$\Phi_{\varepsilon,\alpha}$, the function
		\[
		\varphi_1(t,x):=\phi(t,x)+\psi(t,x,\bar s,\bar y)
		\]
		touches $W_1$ from above at $(\bar t,\bar x)$, while
		\[
		\varphi_2(s,y):=-\psi(\bar t,\bar x,s,y)
		\]
		touches $W_2$ from below at $(\bar s,\bar y)$.
		Hence, they can be taken as test functions for $W_1$ and $W_2$, respectively.
		
		Since $W_1$ is a viscosity subsolution of the HJB equation \eqref{eq:4.1}, and $W_2$ is the corresponding viscosity supersolution,
		by Definition~4.1 applied to the test functions $\varphi_1,\varphi_2$, we have
		\[
		\partial_t \varphi_1(\bar t,\bar x)
		+
		\inf_{u\in U}
		\Bigl\{
		H\bigl(
		\bar t,\bar x,
		\varphi_1(\bar t,\bar x),
		D_x \varphi_1(\bar t,\bar x),
		D^2_x \varphi_1(\bar t,\bar x),
		u
		\bigr)
		\Bigr\}
		\ge 0,
		\]
		and
		\[
		\partial_s \varphi_2(\bar s,\bar y)
		+
		\inf_{u\in U}
		\Bigl\{
		H\bigl(
		\bar s,\bar y,
		\varphi_2(\bar s,\bar y),
		D_y \varphi_2(\bar s,\bar y),
		D^2_y \varphi_2(\bar s,\bar y),
		u
		\bigr)
		\Bigr\}
		\le 0.
		\]
		Here,
		\[
		\partial_t \varphi_1=\partial_t\phi+\partial_t\psi,\quad
		D_x\varphi_1=D_x\phi+D_x\psi,\quad
		D^2_x\varphi_1=X,
		\]
		and
		\[
		\partial_s\varphi_2=-\partial_s\psi,\quad
		D_y\varphi_2=-D_y\psi,\quad
		D^2_y\varphi_2=Y.
		\]
		
		Subtracting the two inequalities and using the fact that
		\(
		\partial_t\psi+\partial_s\psi=0
		\),
		we obtain
		\[
		\partial_t \phi(\bar t, \bar x) + \inf_{u\in U} H_1 - \inf_{u\in U} H_2 \ge 0,
		\tag{5.7}\label{eq:5.7}
		\]
		where
		\[
		H_1 =H\bigl(\bar t,\bar x,\varphi_1(\bar t,\bar x),D_x\varphi_1(\bar t,\bar x),D^2_x \varphi_1(\bar t,\bar x),u\bigr),
		\quad
		H_2 =H\bigl(\bar s,\bar y,\varphi_2(\bar s,\bar y),D_y \varphi_2(\bar s,\bar y),D^2_y \varphi_2(\bar s,\bar y),u\bigr).
		\]
		Furthermore,
		\[
		\inf_{u\in U} H_1- \inf_{u\in U} H_2 \le \sup_{u\in U}  \bigl(  H_1 - H_2 \bigr).
		\tag{5.8}\label{eq:5.8}
		\]
		
		Fix any $u\in U$ and denote
		\[
		F_1 = F(\bar t,\bar x,u),
		\qquad
		F_2 = F(\bar s,\bar y,u).
		\]
		
		Then
		\[
		\begin{aligned}
			H_1 - H_2
			&= m_1\bigl(F_1^\top D_x\varphi_1 - F_2^\top D_y\varphi_2\bigr) \\
			&\quad + \frac12 \sigma^2\bigl(F_1^\top D^2_x \varphi_1 F_1 - F_2^\top D^2_y \varphi_2 F_2\bigr) \\
			&\quad + \int_{\mathbb{R}}
			\Big[ \varphi_1(\bar t,\bar x + F_1\zeta) - \varphi_1(\bar t,\bar x)
			- D_x\varphi_1^\top F_1\zeta \Big] \,\nu(d\zeta) \\
			&\quad - \int_{\mathbb{R}}
			\Big[ \varphi_2(\bar s,\bar y + F_2\zeta) - \varphi_2(\bar s,\bar y)
			- D_y\varphi_2^\top F_2\zeta \Big] \,\nu(d\zeta) \\
			&\quad + f\bigl(\bar t,\bar x,\varphi_1(\bar t,\bar x),(\mathcal{L}^{u,(k)}\varphi_1(\bar t,\bar x))_{k\ge1},u\bigr) \\
			&\quad - f\bigl(\bar s,\bar y,\varphi_2(\bar s,\bar y),(\mathcal{L}^{u,(k)}\varphi_2(\bar s,\bar y))_{k\ge1},u\bigr).
		\end{aligned}
		\tag{5.9}\label{eq:5.9}
		\]
		
		We estimate each term in $H_1-H_2$ in order to relate $\mathcal{L}^u \varphi_1(\bar t,\bar x)-\mathcal{L}^u \varphi_2(\bar s,\bar y)$ to $\mathcal{L}^u \phi(\bar t,\bar x)$.
		
		From the gradient estimate of the penalization function, we have
		\[
		|D_x\psi|+|D_y\psi|
		\le C\!\left(\frac{|x-y|}{\varepsilon^2}+\eta(|x|+|y|)\right).
		\]
		Moreover, at the maximum point $(\bar t,\bar x,\bar s,\bar y)$, by \eqref{eq:5.5} it holds that
		\[
		|\bar x-\bar y|\le C\varepsilon,\qquad
		|\bar t-\bar s|\le C\alpha .
		\]
		
		Note that
		\[
		F_1^{\top}D_x\varphi_1-F_2^{\top}D_y\varphi_2
		=F_1^{\top}D_x\phi+F_1^{\top}D_x\psi+F_2^{\top}D_y\psi ,
		\]
		and
		\[
		\begin{aligned}
			|F_1^{\top}D_x\psi+F_2^{\top}D_y\psi|
			&=|(F_1-F_2)^{\top}D_x\psi+F_2^{\top}(D_x\psi+D_y\psi)| \\
			&\le |F_1-F_2|\,|D_x\psi|
			+|F_2|(|D_x\psi|+|D_y\psi|).
		\end{aligned}
		\]
		
		By the Lipschitz continuity of $F$ with respect to $(t,x)$ and the linear growth condition, we have
		\[
		|F_1-F_2|\le C(|\bar t-\bar s|+|\bar x-\bar y|),\qquad
		|F_2|\le C(1+|\bar y|).
		\]
		Combining the above estimates and using the boundedness ensured by the penalization term
		\[
		-\eta(|x|^2+|y|^2),
		\]
		we obtain
		\[
		|F_1^{\top}D_x\psi+F_2^{\top}D_y\psi|
		\le C\!\left(\frac{|\bar x-\bar y|^2}{\varepsilon^2}+\eta\right).
		\]
		Consequently,
		\[
		F_1^{\top}D_x\varphi_1-F_2^{\top}D_y\varphi_2
		\le F_1^{\top}D_x\phi
		+ C\!\left(\frac{|\bar x-\bar y|^2}{\varepsilon^2}+\eta\right).
		\]
		
		Multiplying both sides by the constant $m_1$ (and absorbing it into the constant $C$), we derive
		\[
		m_1\bigl(F_1^{\top}D_x\varphi_1-F_2^{\top}D_y\varphi_2\bigr)
		\le m_1 F_1^{\top}D_x\phi
		+ C\!\left(\frac{|\bar x-\bar y|^2}{\varepsilon^2}+\eta\right).
		\tag{5.10}\label{eq:5.10}
		\]
		
		By the matrix inequality \eqref{eq:5.6}, we have
		\[
		F_1^\top D^2_x \varphi_1 F_1 - F_2^\top D^2_y \varphi_2 F_2
		\le
		\frac{2}{\varepsilon^2}|F_1-F_2|^2
		+2\eta\bigl(|F_1|^2+|F_2|^2\bigr).
		\]
		Using the Lipschitz continuity of $F$, we obtain
		\[
		F_1^\top D^2_x \varphi_1 F_1 - F_2^\top D^2_y \varphi_2 F_2
		\le C\left(\frac{|\bar x-\bar y|^2}{\varepsilon^2}+\eta\right).
		\tag{5.11}\label{eq:5.11}
		\]
		
		By the maximality of $\Phi_{\varepsilon,\alpha}$ at $(\bar t,\bar x,\bar s,\bar y)$, for all $\zeta \in \mathbb{R}$,
		\[
		\Phi_{\varepsilon,\alpha}(\bar t,\bar x + F_1 \zeta, \bar s, \bar y + F_2 \zeta)
		\le \Phi_{\varepsilon,\alpha}(\bar t, \bar x, \bar s, \bar y).
		\]
		
		Hence,
		\[
		\begin{aligned}
			&W_1(\bar t, \bar x + F_1 \zeta) - W_1(\bar t, \bar x) - \bigl(W_2(\bar s, \bar y + F_2 \zeta) - W_2(\bar s, \bar y)\bigr) \\
			&\quad \le \phi(\bar t, \bar x + F_1 \zeta) - \phi(\bar t, \bar x)
			+ \psi(\bar t, \bar x + F_1 \zeta, \bar s, \bar y + F_2 \zeta) - \psi(\bar t, \bar x, \bar s, \bar y).
		\end{aligned}
		\]
		
		Define
		\[
		\varphi_1(t,x) = \phi(t,x) + \psi(t,x,\bar s,\bar y), \qquad
		\varphi_2(s,y) = -\psi(\bar t,\bar x,s,y).
		\]
		
		Then
		\[
		\begin{aligned}
			&W_1(\bar t, \bar x + F_1 \zeta) - W_1(\bar t, \bar x) - \bigl(W_2(\bar s, \bar y + F_2 \zeta) - W_2(\bar s, \bar y)\bigr) \\
			&\quad \le \varphi_1(\bar t, \bar x + F_1 \zeta) - \varphi_1(\bar t, \bar x)
			- \varphi_2(\bar s, \bar y + F_2 \zeta) + \varphi_2(\bar s, \bar y).
		\end{aligned}
		\]
		
		Integrating with respect to $\nu(d\zeta)$, we have
		\[
		\begin{aligned}
			&\int_{\mathbb{R}} \bigl[ W_1(\bar t, \bar x + F_1 \zeta) - W_1(\bar t, \bar x) - D_x \varphi_1^\top F_1 \zeta \bigr] \, \nu(d\zeta) \\
			&\quad - \int_{\mathbb{R}} \bigl[ W_2(\bar s, \bar y + F_2 \zeta) - W_2(\bar s, \bar y) - D_y \varphi_2^\top F_2 \zeta \bigr] \, \nu(d\zeta) \\
			&\le \int_{\mathbb{R}} \bigl[ \varphi_1(\bar t, \bar x + F_1 \zeta) - \varphi_1(\bar t, \bar x) - D_x \varphi_1^\top F_1 \zeta \bigr] \, \nu(d\zeta) \\
			&\quad - \int_{\mathbb{R}} \bigl[ \varphi_2(\bar s, \bar y + F_2 \zeta) - \varphi_2(\bar s, \bar y) - D_y \varphi_2^\top F_2 \zeta \bigr] \, \nu(d\zeta).
		\end{aligned}
		\]
		
		Since $\varphi_1,\varphi_2 \in C^{1,2}$, for any $\zeta\in\mathbb{R}$, by the second-order Taylor expansion with respect to the space variables, there exist $\theta_1,\theta_2\in(0,1)$ such that
		\[
		\varphi_1(\bar t,\bar x+F_1\zeta)-\varphi_1(\bar t,\bar x)-D_x\varphi_1(\bar t,\bar x)^\top F_1\zeta
		=\frac12 (F_1\zeta)^\top D_{xx}^2\varphi_1(\bar t,\bar x+\theta_1F_1\zeta)(F_1\zeta),
		\]
		and
		\[
		\varphi_2(\bar s,\bar y+F_2\zeta)-\varphi_2(\bar s,\bar y)-D_y\varphi_2(\bar s,\bar y)^\top F_2\zeta
		=\frac12 (F_2\zeta)^\top D_{yy}^2\varphi_2(\bar s,\bar y+\theta_2F_2\zeta)(F_2\zeta).
		\]
		Therefore, the difference of the integrands can be estimated by
		\[
		\frac12 |F_1\zeta|^2 \|D_{xx}^2\varphi_1\| + \frac12 |F_2\zeta|^2 \|D_{yy}^2\varphi_2\|.
		\]
		
		By the explicit form of $\psi$ in \eqref{eq:5.4}, we have
		\[
		D_{xx}^2\psi=\frac{2}{\varepsilon^2}I+2\eta I,\qquad
		D_{yy}^2\psi=\frac{2}{\varepsilon^2}I+2\eta I.
		\]
		Consequently,
		\[
		\|D_{xx}^2\varphi_1\|\le \|D_{xx}^2\phi\|+\frac{C}{\varepsilon^2}+C\eta,\qquad
		\|D_{yy}^2\varphi_2\|\le \frac{C}{\varepsilon^2}+C\eta.
		\]
		Using the boundedness of the penalized maximum points, we obtain
		\[
		\|D_{xx}^2\varphi_1\|+\|D_{yy}^2\varphi_2\|
		\le C\Bigl(\frac{|\bar x-\bar y|^2}{\varepsilon^2}+\eta+1\Bigr).
		\]
		
		Combining the above estimates and using the assumption
		$\int_{\mathbb{R}} |\zeta|^2 \,\nu(d\zeta) < \infty$, we obtain
		\begin{equation}\tag{5.12}\label{eq:5.12}
			\begin{aligned}
				&\int_{\mathbb{R}}
				\Big[ \varphi_1(\bar t,\bar x + F_1\zeta) - \varphi_1(\bar t,\bar x)
				- D_x\varphi_1^\top F_1\zeta \Big] \,\nu(d\zeta)  \\
				&\quad - \int_{\mathbb{R}}
				\Big[ \varphi_2(\bar s,\bar y + F_2\zeta) - \varphi_2(\bar s,\bar y)
				- D_y\varphi_2^\top F_2\zeta \Big] \,\nu(d\zeta) \\
				&\le C \left(\frac{|\bar x-\bar y|^2}{\varepsilon^2}+\eta\right).
			\end{aligned}
		\end{equation}
		
		Combining \eqref{eq:5.10}, \eqref{eq:5.11} and \eqref{eq:5.12}, we deduce that
		\[
		\mathcal{L}^u \varphi_1(\bar t,\bar x)-\mathcal{L}^u \varphi_2(\bar s,\bar y)
		\le
		\mathcal{L}^u \phi(\bar t,\bar x)
		+ C\left(\frac{|\bar x-\bar y|^2}{\varepsilon^2}+\eta\right).
		\tag{5.13}\label{eq:5.13}
		\]
		
		For the $f$-term, we decompose the difference as follows:
		\[
		\begin{aligned}
			& f\bigl(\bar t,\bar x,\varphi_1(\bar t,\bar x),(\mathcal{L}^{u,(k)}\varphi_1(\bar t,\bar x))_{k\ge1},u\bigr)
			- f\bigl(\bar s,\bar y,\varphi_2(\bar s,\bar y),(\mathcal{L}^{u,(k)}\varphi_2(\bar s,\bar y))_{k\ge1},u\bigr) \\
			= {} &
			\Bigl[
			f\bigl(\bar t,\bar x,\varphi_1(\bar t,\bar x),(\mathcal{L}^{u,(k)}\varphi_1(\bar t,\bar x))_{k\ge1},u\bigr)
			- f\bigl(\bar t,\bar x,\varphi_2(\bar s,\bar y),(\mathcal{L}^{u,(k)}\varphi_2(\bar s,\bar y))_{k\ge1},u\bigr)
			\Bigr] \\
			&+
			\Bigl[
			f\bigl(\bar t,\bar x,\varphi_2(\bar s,\bar y),(\mathcal{L}^{u,(k)}\varphi_2(\bar s,\bar y))_{k\ge1},u\bigr)
			- f\bigl(\bar s,\bar y,\varphi_2(\bar s,\bar y),(\mathcal{L}^{u,(k)}\varphi_2(\bar s,\bar y))_{k\ge1},u\bigr)
			\Bigr].
		\end{aligned}
		\tag{5.14}\label{eq:5.14}
		\]
		
		By the Lipschitz continuity of $f$ with respect to $(y,z)$ in Assumption~\ref{ass:3.1} and the linearity of
		$\mathcal{L}^{u,(k)}$, we have
		\[
		\begin{aligned}
			& f\bigl(\bar t,\bar x,\varphi_1(\bar t,\bar x),(\mathcal{L}^{u,(k)}\varphi_1(\bar t,\bar x))_{k\ge1},u\bigr)
			- f\bigl(\bar t,\bar x,\varphi_2(\bar s,\bar y),(\mathcal{L}^{u,(k)}\varphi_2(\bar s,\bar y))_{k\ge1},u\bigr) \\
			&\quad \le
			L_2 \bigl|\varphi_1(\bar t,\bar x)-\varphi_2(\bar s,\bar y)\bigr|
			+ L_3
			\Bigl\|
			(\mathcal{L}^{u,(k)}(\varphi_1-\varphi_2)(\bar t,\bar x))_{k\ge1}
			\Bigr\|_{\ell^2}.
		\end{aligned}
		\tag{5.15}\label{eq:5.15}
		\]
		
		By the continuity of $f$ with respect to $(t,x)$ and the properties of the penalization function \eqref{eq:5.5}, we have
		\[
		f(\bar t,\bar x,\varphi_2(\bar s,\bar y),(\mathcal L^{u,(k)}\varphi_2(\bar s,\bar y))_{k\ge1},u)
		- f(\bar s,\bar y,\varphi_2(\bar s,\bar y),(\mathcal L^{u,(k)}\varphi_2(\bar s,\bar y))_{k\ge1},u)
		\le C\!\left(\frac{|\bar x-\bar y|^2}{\varepsilon^2}+\eta\right).
		\tag{5.16}\label{eq:5.16}
		\]
		
		From \eqref{eq:5.13}, \eqref{eq:5.15} and \eqref{eq:5.16}, it follows that for any $u \in U$,
		\[
		H_1 - H_2
		\le
		\mathcal{L}^u \phi(\bar t, \bar x)
		+ L_2 \bigl|\varphi_1(\bar t,\bar x)-\varphi_2(\bar s,\bar y)\bigr|
		+ L_3
		\Bigl\|
		(\mathcal{L}^{u,(k)}(\varphi_1-\varphi_2)(\bar t,\bar x))_{k\ge1}
		\Bigr\|_{\ell^2}
		+ C\left(\frac{|\bar x-\bar y|^2}{\varepsilon^2}+\eta\right).
		\]
		
		Since
		\[
		\varphi_1(\bar t,\bar x)-\varphi_2(\bar s,\bar y)
		=
		\phi(\bar t,\bar x)
		+\psi(\bar t,\bar x,\bar s,\bar y),
		\]
		and
		\[
		|\bar x-\bar y|\to0, \qquad |\bar t-\bar s|\to0
		\quad \text{as } \varepsilon,\alpha\to0,
		\]
		letting $\varepsilon,\alpha,\eta\to0$ yields
		\[
		H_1 - H_2
		\le
		\mathcal{L}^u \phi(t_0,x_0)
		+ L_2 \, |\phi(t_0,x_0)|
		+ L_3 \Big\| \big( \mathcal{L}^{u,(k)} \phi(t_0,x_0) \big)_{k\ge1} \Big\|_{\ell^2}.
		\]
		
		Combining this with \eqref{eq:5.7} and \eqref{eq:5.8}, we conclude that
		\[
		\partial_t \phi(t_0,x_0)
		+
		\sup_{u\in U}
		\Bigl\{
		\mathcal{L}^u \phi(t_0,x_0)
		+ L_2|\phi(t_0,x_0)|
		+ L_3\|(\mathcal{L}^{u,(k)}\phi(t_0,x_0))_{k\ge1}\|_{\ell^2}
		\Bigr\}
		\ge 0.
		\]
		
		Therefore, $\omega=W_1-W_2$ is a viscosity subsolution of the auxiliary equation \eqref{eq:5.2}.
	\end{proof}
	
	Now we can prove the uniqueness theorem.
	
	\begin{thm}\label{thm:5.3}
		Under Assumption~\ref{ass:3.1}, suppose that the L\'{e}vy measure $\nu$ satisfies a second-moment condition, and that $W_1, W_2 \in C([0,T]\times \mathbb{R}^n)$ with linear growth.
		If $W_1$ is a viscosity subsolution of the HJB equation \eqref{eq:4.1}, and $W_2$ is a viscosity supersolution of the HJB equation \eqref{eq:4.1}, then for any $(t,x)\in [0,T]\times \mathbb{R}^n$,
		\[
		W_1(t,x) \le W_2(t,x).
		\]
		
		In particular, if both $W_1$ and $W_2$ are viscosity solutions of the HJB equation \eqref{eq:4.1}, then
		\[
		W_1(t,x) = W_2(t,x), \quad \forall (t,x) \in [0,T]\times \mathbb{R}^n,
		\]
		that is, the viscosity solution of the HJB equation \eqref{eq:4.1} is unique in this class of functions.
	\end{thm}
	
	\begin{lem}\label{lem:5.4}
		Let $\omega(t,x)$ be a viscosity subsolution of the auxiliary equation \eqref{eq:5.2}.
		Define the decaying auxiliary function
		\[
		\tilde \omega(t,x) = e^{\lambda t} \omega(t,x) - \delta t - \varepsilon |x|^2,
		\quad \delta>0, \ \lambda > L_2, \ \varepsilon>0.
		\tag{5.17}\label{eq:5.17}
		\]
		If there exists $(t_0,x_0) \in [0,T) \times \mathbb{R}^n$ such that $\omega(t_0,x_0)>0$,
		then the local  maximum point $(\bar t, \bar x)$ of $\tilde \omega$ must lie in the interior region $(0,T) \times \mathbb{R}^n$,
		and there exists a smooth test function $\varphi$ touching $\omega$ from above at $(\bar t,\bar x)$ such that
		\[
		\partial_t\varphi(\bar t,\bar x)
		+
		\sup_{u\in U}
		\Bigl\{
		\mathcal {L}^u\varphi(\bar t,\bar x)
		+
		L_2|\varphi(\bar t,\bar x)|
		+
		L_3
		\bigl\|(\mathcal{L}^{u,(k)}\varphi(\bar t,\bar x))_{k\ge1}\bigr\|_{\ell^2}
		\Bigr\}
		<0.
		\tag{5.18}\label{eq:5.18}
		\]
	\end{lem}
	
	\begin{proof}
		Assume that there exists $(t_0,x_0)\in[0,T)\times\mathbb{R}^n$ such that $\omega(t_0,x_0)>0$.
		Define the decaying auxiliary function
		\[
		\tilde{\omega}(t,x)
		:=
		e^{\lambda t}\omega(t,x)
		-
		\delta t
		-
		\varepsilon |x|^2,
		\]
		where $\delta>0$ and $\varepsilon>0$ are penalty parameters, and choose
		$\lambda > L_2$.
		
		Since $\omega\in C\bigl([0,T]\times\mathbb{R}^n\bigr)$ and satisfies a linear growth condition,
		the auxiliary function $\tilde{\omega}$ contains the quadratic penalty term $-\varepsilon|x|^2$.
		Hence, for any $t\in[0,T]$,
		\[
		\lim_{|x|\to\infty}\tilde{\omega}(t,x)=-\infty.
		\]
		This shows that the local maximum of $\tilde{\omega}$ cannot be attained at spatial infinity.
		
		Furthermore, by the terminal condition $\omega(T,x)=W_1(T,x)-W_2(T,x)=\phi(x)-\phi(x)=0$,
		the linear decay term $-\delta t$ guarantees that at the terminal time
		\[
		\tilde{\omega}(T,x)
		=
		e^{\lambda T}\omega(T,x)
		-
		\delta T
		-
		\varepsilon |x|^2
		=
		-\delta T
		-
		\varepsilon |x|^2
		<
		0,
		\qquad x\in\mathbb{R}^n.
		\]
		
		On the other hand, by the assumption that there exists $(t_0,x_0)$ such that $\omega(t_0,x_0)>0$,
		we may choose $\delta>0$ and $\varepsilon>0$ sufficiently small so that
		\[
		\tilde{\omega}(t_0,x_0)
		=
		e^{\lambda t_0}\omega(t_0,x_0)
		-
		\delta t_0
		-
		\varepsilon |x_0|^2
		>
		0.
		\]
		
		By the continuity of $\tilde{\omega}$, its local maximum must be attained in a bounded region.
		Denote the local maximum by
		\[
		M := \tilde{\omega}(\bar t,\bar x)
		\ge
		\tilde{\omega}(t_0,x_0)
		>
		0.
		\]
		Therefore, the values of $\tilde{\omega}$ on the time boundary $t=T$ and at spatial infinity
		are strictly less than $M$, and hence the local maximum point $(\bar t,\bar x)$ must satisfy
		\[
		(\bar t,\bar x)\in(0,T)\times\mathbb{R}^n.
		\]
		
		At this point, $\tilde{\omega}$ attains its local maximum at $(\bar t,\bar x)$, which is equivalent to
		\[
		e^{\lambda t}\omega(t,x)-\delta t-\varepsilon|x|^2
		\le
		e^{\lambda\bar t}\omega(\bar t,\bar x)-\delta\bar t-\varepsilon|\bar x|^2,
		\quad
		\forall (t,x)\in[0,T]\times\mathbb{R}^n.
		\]
		Rearranging yields
		\[
		\omega(t,x)
		\le
		e^{-\lambda (t-\bar t)}\,\omega(\bar t,\bar x)
		+
		e^{-\lambda t}
		\Bigl[
		\delta (t-\bar t)
		+
		\varepsilon\bigl(|x|^{2}-|\bar x|^{2}\bigr)
		\Bigr],
		\qquad
		\forall (t,x)\in[0,T]\times\mathbb{R}^{n}.
		\]
		
		Construct the smooth test function
		\[
		\varphi(t,x)
		:=
		e^{-\lambda (t-\bar t)}\,\omega(\bar t,\bar x)
		+
		e^{-\lambda t}
		\Bigl[
		\delta (t-\bar t)
		+
		\varepsilon\bigl(|x|^{2}-|\bar x|^{2}\bigr)
		\Bigr].
		\tag{5.19}\label{eq:5.19}
		\]
		Clearly, $\varphi\in C^\infty([0,T]\times\mathbb{R}^n)$ and touches $\omega$ from above at $(\bar t,\bar x)$,
		and hence it can be used as a test function in the definition of viscosity solutions.
		
		We now compute the first- and second-order derivatives of $\varphi$ at the contact point $(\bar t,\bar x)$.
		
		Differentiating with respect to $t$, we obtain
		\[
		\partial_t \varphi(t,x)
		=
		-\lambda e^{-\lambda (t-\bar t)} \omega(\bar t,\bar x)
		-
		\lambda e^{-\lambda t}
		\left[
		\delta (t-\bar t)
		+
		\varepsilon\bigl(|x|^2-|\bar x|^2\bigr)
		\right]
		+
		\delta e^{-\lambda t}.
		\]
		At $(t,x)=(\bar t,\bar x)$, since $t-\bar t=0$ and $|x|^2-|\bar x|^2=0$, the bracketed term vanishes, and thus
		\begin{equation}
			\partial_t \varphi(\bar t,\bar x)
			=
			-\lambda \omega(\bar t,\bar x)
			+
			\delta e^{-\lambda \bar t}.
			\tag{5.20}\label{eq:5.20}
		\end{equation}
		
		Taking the gradient with respect to $x$, we have
		\[
		D_x\varphi(t,x)
		=
		2\varepsilon e^{-\lambda t} x,
		\]
		and hence
		\begin{equation}
			D_x\varphi(\bar t,\bar x)
			=
			2\varepsilon e^{-\lambda \bar t}\,\bar x.
			\tag{5.21}\label{eq:5.21}
		\end{equation}
		
		Furthermore, the second-order derivative is given by
		\[
		D^2_x\varphi(t,x)
		=
		2\varepsilon e^{-\lambda t} I,
		\]
		where $I$ denotes the $n\times n$ identity matrix. Therefore,
		\begin{equation}
			D^2_x\varphi(\bar t,\bar x)
			=
			2\varepsilon e^{-\lambda \bar t} I.
			\tag{5.22}\label{eq:5.22}
		\end{equation}
		
		Based on \eqref{eq:5.20}--\eqref{eq:5.22}, the first- and second-order spatial derivatives of the test function
		$\varphi$ at the contact point $(\bar t,\bar x)$ can be written explicitly.
		Substituting them into the L\'{e}vy-type generator $\mathcal {L}^u$ and using Assumption~\ref{ass:3.1},
		Lemma~\ref{lem:3.3} and Lemma~\ref{lem:5.1}, which provide uniform boundedness and linear growth properties of the operators,
		we can estimate $\mathcal {L}^u \varphi(\bar t,\bar x)$.
		
		Moreover, since the first- and second-order spatial derivatives of $\varphi$ at $(\bar t,\bar x)$
		depend only on the penalty parameter $\varepsilon$, the coordinates $\bar x$, and the exponential factor
		$e^{-\lambda \bar t}$, and are independent of the value of $\omega$ at the contact point,
		all the terms produced by applying the generator $\mathcal {L}^u$ to $\varphi$ explicitly contain the factor
		$\varepsilon$.
		
		Consequently, there exist constants independent of $\varepsilon$ such that for any $u\in U$,
		\[
		\bigl|\mathcal {L}^u \varphi(\bar t,\bar x)\bigr|
		\le r_1(\varepsilon),
		\qquad
		\bigl\|(\mathcal{L}^{u,(k)}\varphi(\bar t,\bar x))_{k\ge1}\bigr\|_{\ell^2}
		\le r_2(\varepsilon),
		\]
		where the remainder terms $r_i(\varepsilon)$ $(i=1,2)$ satisfy
		\[
		r_i(\varepsilon)=O(\varepsilon),
		\qquad
		r_i(\varepsilon)\to0
		\quad \text{as } \varepsilon\to0.
		\]
		
		Let $r(\varepsilon):=r_1(\varepsilon)+r_2(\varepsilon)$. Then
		\[
		r(\varepsilon)\longrightarrow0,
		\qquad \varepsilon\to0.
		\]
		On the other hand, from the previous derivation, at the contact point $(\bar t,\bar x)$ we have
		\[
		\omega(\bar t,\bar x)=\varphi(\bar t,\bar x)>0,
		\]
		and hence
		\[
		\lvert \varphi(\bar t,\bar x)\rvert
		=
		\varphi(\bar t,\bar x)
		=
		\omega(\bar t,\bar x).
		\]
		Therefore,
		\begin{align}
			&\partial_t \varphi(\bar t,\bar x)
			+
			\sup_{u\in U}
			\Bigl\{
			\mathcal {L}^u \varphi(\bar t,\bar x)
			+
			L_2 \varphi(\bar t,\bar x)
			+
			L_3
			\bigl\|
			(\mathcal{L}^{u,(k)}\varphi(\bar t,\bar x))_{k\ge1}
			\bigr\|_{\ell^2}
			\Bigr\}
			\nonumber\\
			&\le
			-\lambda \omega(\bar t,\bar x)
			+
			\delta e^{-\lambda\bar t}
			+
			L_2\omega(\bar t,\bar x)
			+
			r(\varepsilon)
			\nonumber\\
			&=
			-\bigl(\lambda-L_2\bigr)\omega(\bar t,\bar x)
			+
			\delta e^{-\lambda\bar t}
			+
			r(\varepsilon).
			\tag{5.23}\label{eq:5.23}
		\end{align}
		By the definition of the local maximum $M = \tilde\omega(\bar t,\bar x) \ge \tilde\omega(t_0,x_0)>0$ and the expression of $\tilde\omega$,
		\[
		\tilde\omega(\bar{t},\bar{x}) = e^{\lambda \bar{t}} \omega(\bar{t},\bar{x}) - \delta \bar{t} - \varepsilon |\bar{x}|^2,
		\]
		we obtain
		\[
		e^{\lambda \bar{t}} \omega(\bar{t},\bar{x}) = \tilde\omega(\bar{t},\bar{x}) + \delta \bar{t} + \varepsilon |\bar{x}|^2 \ge M \ge \tilde\omega(t_0,x_0).
		\tag{5.24}\label{eq:5.24}
		\]
		
		Since $\bar{t} \in (0,T)$ and the exponential function $e^{\lambda t}$ is increasing in $t$, we have
		$e^{\lambda \bar{t}} \le e^{\lambda T}$, that is, $e^{-\lambda \bar{t}} \ge e^{-\lambda T}$.
		From \eqref{eq:5.24}, it follows that
		\[
		\omega(\bar{t},\bar{x}) \ge e^{-\lambda \bar{t}} \, \tilde\omega(t_0,x_0) \ge e^{-\lambda T} \, \tilde\omega(t_0,x_0) > 0.
		\]
		
		Therefore, since $\lambda > L_2$, there exists a constant
		\[
		\eta_0:=(\lambda - L_2 )\, e^{-\lambda T}\, \tilde\omega(t_0,x_0)> 0,
		\]
		independent of $\varepsilon$ and $\delta$, such that
		\[
		(\lambda - L_2)\, \omega(\bar{t},\bar{x}) \ge \eta_0 > 0.
		\]
		
		Fixing this $\lambda$, we first choose $\varepsilon>0$ sufficiently small and then choose $\delta>0$ sufficiently small so that
		\[
		\delta e^{-\lambda\bar t}+r(\varepsilon)<\frac{\eta_0}{2}.
		\]
		Then, by \eqref{eq:5.23},
		\[
		\partial_t \varphi(\bar t,\bar x)
		+
		\sup_{u\in U}
		\Bigl\{
		\mathcal {L}^u \varphi(\bar t,\bar x)
		+
		L_2 \varphi(\bar t,\bar x)
		+
		L_3
		\bigl\|
		(\mathcal{L}^{u,(k)}\varphi(\bar t,\bar x))_{k\ge1}
		\bigr\|_{\ell^2}
		\Bigr\}
		<-\frac{\eta_0}{2}
		<0,
		\]
		which yields \eqref{eq:5.18}.
	\end{proof}
	
	\begin{proof}[Proof of Theorem~\ref{thm:5.3}]
		Assume that the conclusion does not hold.
		Then there exists $(t_0,x_0)\in[0,T)\times\mathbb{R}^n$ such that
		\[
		W_1(t_0,x_0)-W_2(t_0,x_0)>0.
		\]
		Define the difference function
		\[
		\omega := W_1-W_2.
		\]
		
		By Lemma~\ref{lem:5.2}, $\omega$ is a viscosity subsolution of the auxiliary equation \eqref{eq:5.2},
		and it satisfies the terminal condition
		\[
		\omega(T,x)=0,\qquad x\in\mathbb{R}^n.
		\]
		
		Since there exists $(t_0,x_0)$ such that $\omega(t_0,x_0)>0$,
		together with the above terminal condition,
		Lemma~\ref{lem:5.4} ensures that there exist parameters $\lambda,\delta,\varepsilon>0$ such that the decaying auxiliary function
		\[
		\tilde\omega(t,x)
		:=
		e^{\lambda t}\omega(t,x)-\delta t-\varepsilon|x|^2
		\]
		attains its local maximum at some point $(\bar t,\bar x)\in(0,T)\times\mathbb{R}^n$.
		
		Furthermore, Lemma~\ref{lem:5.4} constructs, by using the extremal property at $(\bar t,\bar x)$,
		a smooth function $\varphi\in C^{1,2}_b$ which touches $\omega$ from above at $(\bar t,\bar x)$,
		that is,
		\[
		\omega(\bar t,\bar x)=\varphi(\bar t,\bar x),
		\qquad
		\omega(t,x)\le \varphi(t,x)
		\quad \text{for all } (t,x)\in [0,T]\times\mathbb{R}^n,
		\]
		and this test function satisfies the strict inequality
		\[
		\partial_t\varphi(\bar t,\bar x)
		+
		\sup_{u\in U}
		\Bigl\{
		\mathcal {L}^u\varphi(\bar t,\bar x)
		+
		L_2|\varphi(\bar t,\bar x)|
		+
		L_3
		\bigl\|(\mathcal{L}^{u,(k)}\varphi(\bar t,\bar x))_{k\ge1} \bigr\|_{\ell^2}
		\Bigr\}
		<0.
		\]
		
		However, since $\omega$ is a viscosity subsolution of the auxiliary equation \eqref{eq:5.2},
		the above function $\varphi$, as a valid test function, must satisfy at $(\bar t,\bar x)$ the inequality
		\[
		\partial_t\varphi(\bar t,\bar x)
		+
		\sup_{u\in U}
		\Bigl\{
		\mathcal {L}^u\varphi(\bar t,\bar x)
		+
		L_2|\varphi(\bar t,\bar x)|
		+
		L_3
		\bigl\|
		(\mathcal{L}^{u,(k)}\varphi(\bar t,\bar x))_{k\ge1}
		\bigr\|_{\ell^2}
		\Bigr\}
		\ge 0,
		\]
		which contradicts the previous strict inequality.
		
		Therefore, the assumption is false, and hence
		\[
		W_1(t,x)\le W_2(t,x),
		\qquad \forall (t,x)\in[0,T]\times\mathbb{R}^n.
		\]
		
		If both $W_1$ and $W_2$ are viscosity solutions of \eqref{eq:4.1}, then applying the above comparison result to $(W_1,W_2)$ and $(W_2,W_1)$ yields
		\[
		W_1(t,x)=W_2(t,x),
		\qquad \forall (t,x)\in[0,T]\times\mathbb{R}^n.
		\]
		Therefore, the viscosity solution of the HJB equation \eqref{eq:4.1} is unique in this class of functions.
	\end{proof}
	
	\section{Conclusion}
	In this paper, we have developed a unified analytical framework for stochastic optimal control of FBSDEs driven by general L\'{e}vy processes, leveraging Teugels martingales to handle infinite-dimensional jump perturbations. The dynamic programming principle is rigorously established, and the associated Hamilton–Jacobi–Bellman equation is shown to admit a unique viscosity solution. These results extend existing stochastic control theory beyond Brownian or Poisson-driven systems, providing a systematic approach for analyzing general L\'{e}vy jump dynamics. The proposed framework lays a foundation for future studies on more complex jump-diffusion control problems and their applications in finance and engineering.

\end{document}